\documentclass[a4paper,reqno,11pt]{amsart}

\usepackage[english]{babel}
\usepackage{amsmath}
\usepackage{amssymb}
\usepackage{amsthm}
\usepackage{enumerate}
\usepackage{ifthen}
\usepackage{bbm}
\usepackage{graphicx,subfigure}  

\usepackage{marginnote} %

\usepackage[a4paper,left=35mm,right=35mm,top=35mm,bottom=35mm,marginpar=25mm]{geometry}

\usepackage{color}
\usepackage{epsfig}
\usepackage{hyperref}
\usepackage{dsfont}

\usepackage{xpatch}
\makeatletter
\patchcmd{\@tocline}{\hfil}
{\nobreak\leaders\hbox{\ifnum#1<2\hfill\else$\m@th%
\mkern 4.5 mu\hbox{.}\mkern 4.5 mu$\fi}\hfill\nobreak}{}{}
\def\l@section{\@tocline{1}{10pt}{1pc}{}{\bfseries}}
\def\l@subsection{\@tocline{2}{0pt}{\dimexpr 1pc+2em}{}{}}
\makeatother

\newtheorem{theorem}{Theorem}

\newtheorem{proposition}[theorem]{Proposition}
\newtheorem{lemma}[theorem]{Lemma}

\newtheorem{remark}[theorem]{Remark}
\theoremstyle{plain}

\newtheorem*{theorem*}{Theorem}

\allowdisplaybreaks

\newcommand{\del}{\partial}
\newcommand{\eps}{\varepsilon}

\newcommand{\cD}{\mathcal{D}}

\newcommand{\R}{\mathbb{R}}

\begin{document}

\title[Dispersive Shocks]{Dispersive shocks in diffusive-dispersive approximations of elasticity and quantum-hydrodynamics}

\author[D. Bolbot]{Daria Bolbot}
\address[Daria Bolbot]{
\newline
Computer, Electrical and Mathematical Science and Engineering Division 
\newline
King Abdullah University of Science and Technology (KAUST)
\newline 
Thuwal 23955-6900,  Saudi Arabia
}
\email{daria.bolbot@kaust.edu.sa}
\author[D. Mitsotakis]{Dimitrios Mitsotakis}
\address[Dimitrios Mitsotakis]{\newline
Victoria University of Wellington
\newline School of Mathematics and Statistics
\newline
PO Box 600
\newline
Wellington 6140
\newline 
New Zealand
}
\email{dimitrios.mitsotakis@vuw.ac.nz}
\author[A.E. Tzavaras]{Athanasios E. Tzavaras}
\address[Athanasios E. Tzavaras]{
\newline
Computer, Electrical and Mathematical Science and Engineering Division 
\newline
King Abdullah University of Science and Technology (KAUST)
\newline 
Thuwal 23955-6900,  Saudi Arabia
}
\email{athanasios.tzavaras@kaust.edu.sa}

\dedicatory{To Constantine Dafermos, who keeps inspiring us, with friendship and admiration}

\begin{abstract}
The aim is to assess the combined effect of diffusion and dispersion on shocks in the moderate dispersion regime.
For a diffusive dispersive approximation of the equations of one-dimensional elasticity (or p-system), we study convergence 
of traveling waves to shocks. The problem is recast as a Hamiltonian system with small friction, and an analysis of the length
of oscillations yields convergence in the moderate dispersion regime $\eps, \delta \to 0$ with $\delta = o(\eps)$, under hypotheses
that the limiting shock is admissible according to the Liu E-condition and is not a contact  discontinuity at either end state. 
A similar convergence result is proved for traveling waves of the quantum hydrodynamic system with artificial viscosity 
as well as for a viscous Peregrine-Boussinesq system where traveling waves model undular bores, in all cases in the moderate dispersion regime. 
\end{abstract}

\maketitle


\section{Introduction}


Systems exhibiting interplay of diffusion, dispersion and nonlinear response have been extensively studied in the field of conservation laws,
 starting  with works on the subject of phase transitions and  undercompressive shocks, {{\it e.g.} \cite{slemrod1983,SS1987,JMS1995,BL2001,BL2002}.
An alternative perspective arose from the field of nonlinear dispersive equations with the objective to
study dispersive or dissipative-dispersive shocks \cite{BS1985,EHS2017}. Similar problems appear in a variety of fluid mechanics settings, like shallow water  flows \cite{Chanson2009,DHM2018}, undular waves in atmospheric flows or water waves \cite{peregrine1967,BMT2022}.
Peregrine \cite{Pere1966} introduced weakly nonlinear and dispersive wave equations in order to study undular bores, a wave appearing 
in rivers, atmospheric flows and also in blood vessels composed of a solitary wave followed by undulations. 

The Burgers-Korteweg de Vries (KdV) equation
\begin{equation}\label{eq:vbKdV}
u_t +  f(u)_x = \eps u_{xx} -  \delta u_{xxx}\ ,
\end{equation}
has been a testing ground  for assessing the interplay of diffusion, dispersion and nonlinearity. Various perspectives of study exist:
(a) to assess the effect of dispersion in the KdV or modified KdV equation ($\eps =0$) on expanding wavetrain solutions connecting two constant states
 arising via Whitham modulation theory and termed in the field of dispersive equations as dispersive shock waves; (b) to assess the limiting behavior
of traveling wave solutions when both diffusion and dispersion are present. We  refer to \cite{EHS2017} for an in depth presentation of these viewpoints 
and their relation.
Here, we focus on a specific aspect of (b),  relevant from the viewpoint of systems of conservation laws,  namely how oscillatory traveling waves 
for diffusive-dispersive systems of two conservation laws approach shocks in the limit $\eps, \delta \to 0$.

Traveling wave solutions have been a focal point for assessing the interplay of diffusive-dispersive systems with studies for the Burgers-KdV equation
with $f(u) = u^2$ \cite{BS1985}, the modified Burgers-KdV equation with $f(u) = u^3$ \cite{JMS1995}, diffusive-dispersive approximations of elasticity
 \cite{HS1983,B1987,BL2001} or hyperbolic-elliptic models for phase transitions \cite{slemrod1983,BL2002}. Comprehensive presentations 
 can be found in \cite{lef2002,FS2002}. 
The related convergence results from shock profiles to shock waves generally hold in the weak dispersion regime $\delta = O(\eps^2)$. 
Based on such studies and related convergence results from \eqref{eq:vbKdV} to the inviscid Burgers equation  \cite{Schonbek1982,HT2002}
it was believed for a while that $\delta = O(\eps^2)$ might be the threshold for convergence to Kruzhkov entropy solutions. 
This was refuted in \cite{PR2007} where traveling wave solutions for genuinely nonlinear Burgers-KdV equations were shown to converge to shocks in the range $\delta = o(\eps)$. 
In that range traveling waves present relatively strong oscillatory behavior and dispersive effects are significant,
hence it was termed moderate dispersion regime.


The aim here is to examine  the convergence of diffusive-dispersive traveling waves 
for systems of conservation laws in the moderate dispersion regime.
We note that the use of genuine nonlinearity is avoided and replaced by the Liu shock admissibility condition and 
a requirement that the shock is not a (right or left) contact discontinuity.
The analysis is developed for the system of elasticity and extended to other situations where diffusive-dispersive effects play a role: 
the quantum hydrodynamic system with diffusion and to diffusive-dispersive models modeling undular bores.

The specific cases analyzed are the following:
We first consider a diffusive-dispersive approximation of the elasticity system
\begin{equation}
\label{EDDI}
\begin{aligned}
u_t &=v_x\ ,   \\
 v_t &=\sigma(u)_x+\eps v_{xx}-\delta u_{xxx}\ , 
\end{aligned}
\end{equation}
where  $(t,x)\in\mathbb{R_+}\times\mathbb{R}$ and $\eps, \delta > 0$. The hyperbolic part of \eqref{EDDI} is known as the p-system and is expressed in Lagrangian coordinates.
The hyperbolicity condition $\sigma'(u) > 0$ is employed throughout  but use of genuine nonlinearity is avoided.
Existence for traveling waves and convergence to shocks in the range $\delta = O(\eps^2)$ appear in \cite{HS1983,B1987,BL2001} 
using a dynamical systems approach. The convergence from traveling waves to shock waves is extended here to the moderate dispersion regime
 $\delta = o(\eps)$  -- as contrasted to the weak dispersion regime  $\delta = O(\eps^2)$ --
for a shock satisfying the Liu E-condition and avoiding contact discontinuities at the end states.
This convergence covers the regime of moderate dispersion where oscillations have a significant presence.
Our analysis does not cover  undercompressive shocks or non-monotone stress-strain relations appearing in phase transitions 
or Van der Waals fluids;  we refer to \cite{FS2002,lef2002} and  references therein for reviews of those subjects.

Next, consider the Quantum hydrodynamics system with artificial viscosity
\begin{equation}
\label{eq:QHDDDI}
\begin{aligned}
\rho_t+j_x &=0\ ,   \\
 j_t+ \left (\frac{j^2}{\rho}+\rho^\gamma \right)_x &=\eps j_{xx}+\delta\rho \left(\frac{\sqrt{\rho}_{xx}}{\sqrt{\rho}} \right)_x \ .
\end{aligned}
\end{equation}
Here, $\rho=\rho(t,x)>0$ denotes the density, $j=j(t,x)$ momentum, $(t,x) \in\mathbb{R_+}\times\mathbb{R}$ while
$\rho^\gamma$ stands for the pressure with $\gamma\geq1$. This system with $\eps = 0$ is used to model semiconductors or superfluidity, 
the dispersive term $\rho \left(\frac{\sqrt{\rho}_{xx}}{\sqrt{\rho}}\right)_x$ is called quantum Bohm potential \cite{wyatt2005quantum}, while the 
term $\eps j_{xx}$ with $\eps>0$ models artificial viscosity. Traveling wave analysis and convergence to shock waves in the regime $\delta = O(\eps^2)$ 
is performed in \cite{LMZ2020, LZ2021}. It is here improved  by showing that  diffusive-dispersive shock profiles converge to shocks in the regime $\delta = o(\eps)$.

The third example is the dissipative Peregrine-Boussinesq system
\begin{equation}\label{eq:boussinesqs2} 
    \begin{aligned}
        &\eta_t + u_x + (\eta u)_x = 0\ ,\\
        &u_t + \eta_x + uu_x - \delta u_{xxt} - \varepsilon u_{xx} = 0\ .
    \end{aligned}
\end{equation}
We consider traveling wave solutions $(\eta (\tau), u(\tau))$,  $\tau=-\frac{x-st}{\sqrt{s\delta}}$ for $s>1$,  under the limiting conditions
\begin{equation}\label{eq:condition1}
\lim_{\tau \to -\infty} (\eta , u) = (\eta^-,u^-) = (0,0), \quad 
\lim_{\tau \to +\infty} (\eta, u) = (\eta^+,u^+)\ ,
\end{equation}
Here $\eta$ is the elevation of the free surface, and $u$ the horizontal velocity of the fluid measured at some height above a flat bottom.
This model has been used for the prediction of undular bores, \cite{BMT2022},  as a balance of nonlinear shock formation of the shallow water equations
and dispersive effects of water waves.  For some values of Froude number though, the oscillations can disappear and classical shock waves can be formed.
Existence of traveling waves can be found in \cite{BMT2022}; this result is complemented here by convergence to shock waves in the regime $\delta = o (\eps)$.

A key ingredient is the analysis of the length of the oscillatory tail inspired by the approach of \cite{PR2007}.
The traveling wave problem is recast as Hamiltonian system with friction of size  $c=\eps/\sqrt{\delta}s$, with $s$ the shock speed, like in \cite{JMS1995,PR2007}.
The regime of moderate dispersion corresponds to small friction $c < c^*$, where $c^*$ is some critical threshold.  In contrast to  \cite{PR2007}  the genuine nonlinearity hypothesis is
replaced by the use of Liu E-condition for shock admissibility. The main result concerning the size of oscillations is stated in 
Proposition \ref{propsize}. It is used to show convergence of traveling wave solutions of \eqref{EDDI}
to the equations of elasticity in the limit $\eps,\delta \to 0$ with $\delta = o(\eps)$, see Theorem \ref{thm:convergence}.
The same methodology is applied to show convergence from traveling waves to shocks for the quantum hydrodynamics system with artificial viscosity  \eqref{eq:QHDDDI}  
and a similar result for the  Peregrine-Boussinesq system with viscosity \eqref{eq:boussinesqs2}; in both cases in the regime $\delta = o(\eps)$.

The manuscript is organized as follows:  In Section \ref{sec:preliminary} the traveling wave problem for the diffusive-dispersive regularization of the
elasticity system \eqref{EDDI} is reduced to a Hamiltonian system with friction  \eqref{PS2}--\eqref{defPhi3}. 
Moderate dispersion leads to the study of a regime of weak friction, carried out in Section \ref{sec:main}.
The length of the oscillatory tail for traveling wave solutions is estimated in Proposition \ref{propsize}. As a corollary, convergence from
traveling waves to shocks for \eqref{EDDI} holds for $\delta = o(\eps)$, stated in  Theorem  \ref{thm:convergence}. 
In Section \ref{sec:qhd}, the quantum hydrodynamic system with viscosity is considered for genuinely nonlinear pressures. The problem of
traveling waves is recast in the form of the problem \eqref{PS2}--\eqref{defPhi3}, and convergence to Lax shocks is
shown in the moderate dispersion regime $\delta=o(\eps)$, see Theorem \ref{TQHD}.
The dissipative Peregrine-Boussinesq system \eqref{eq:boussinesqs2} is studied in section \ref{sec:pbs} and convergence in the moderate
dispersion regime is again based in Proposition \ref{propsize}.

\section{Diffusion-dispersion approximation of the elasticity system}\label{sec:preliminary}

We consider a diffusive-dispersive approximation for the one-dimensional elasticity system 
\begin{equation}
\label{EDD}
\begin{aligned}
u_t &=v_x\ ,   \\
 v_t &=(\sigma(u))_x+\eps v_{xx}-\delta u_{xxx}\ , 
\end{aligned}
\end{equation}
where $(t,x)\in\mathbb{R_+}\times\mathbb{R}$ are the time and space variables, $u$ is the strain, $v$ the velocity and $\sigma(u)$ the stress depending on $u$. 
The parameters $\eps>0$, $\delta > 0$ in \eqref{EDD} measure respectively the sizes of diffusion and dispersion.
Throughout this work we assume that $\sigma'(u) > 0$.

\subsection{Preliminaries: Shocks for the p-system}\label{eq:psystem}
The system \eqref{EDD} is a regularization of the equations of one-dimensional nonlinear elasticity, also called $p$-system:
\begin{equation}
\label{EL}
\begin{aligned}
\del_t u  &=  \del_x v  \, ,
\\
\del_t v  &=  \del_x \sigma(u)  \;.\
\end{aligned}
\end{equation}
There are interpretations of \eqref{EL} representing both longitudinal and shear motions; for longitudinal motions $u >0$ 
is interpreted as the longitudinal strain, for shear motions $u \in \R$ is a shear strain. When $u$ represents longitudinal motions
$\eps v_{xx}$ needs to be replaced by $(\frac{\eps}{u} v_x )_x$ but we will not consider such issues here focusing on the 
essential behaviors.

 When  $\sigma'(u)>0$, the system \eqref{EL} is strictly hyperbolic 
with wave speeds $\lambda_1=-\sqrt{\sigma'(u)}$, $\lambda_2=\sqrt{\sigma'(u)}$.
Shocks are generated by solving the Rankine-Hugoniot conditions
\begin{equation}\label{rh}
\begin{aligned}
- s (u_+ - u_-) &= (v_+ - v_-)\ ,
\\
- s (v_+ - v_-) &= (\sigma(u_+ ) -\sigma(u_-) )\ ,
\end{aligned}
\end{equation}
where $s$ is the shock speed and $(u_-, v_-)$, $(u_+ , v_+)$ the left and right states, respectively.
The shock speed is computed by 
\begin{equation}\label{shockspeed}
s^2=(\sigma(u_+)-\sigma(u_-))/(u_+-u_-) \; ,
\end{equation}
and \eqref{EL} admits two types of shocks: 
1-shocks with $s<0$ moving backwards and 2-shocks with $s>0$ moving forward.
When  $\sigma''(u) \ne 0$ the system is called genuinely nonlinear; this assumption is not, in general, adopted here
and $\sigma(u)$ will be allowed to have inflection points.

Admissibility conditions are imposed on shocks, motivated by either stability considerations 
or from requesting that admissible shocks emerge as limits of traveling waves for viscosity regularizations. 
We refer to  \cite[Ch VIII]{Dafermos} for an in depth discussion
of shock-admissibility criteria and \cite[Ch 18]{Smoller} for the construction of shock curves and the solution 
of the Riemann problem for \eqref{EL}.

For \eqref{EL}, under the hyperbolicity assumption $\sigma'(u) > 0$, the usual admissibility criteria are:
\begin{itemize}
\item[(i)] For genuinely nonlinear systems $\sigma''(u) \ne 0$ admissible shocks are selected by 
the Lax-shock admissibility criterion, stating that  admissible 1-shocks ($s<0$) satisfy 
\begin{equation}\label{shockL1}
-\sqrt{\sigma'(u_+)}<s<-\sqrt{\sigma'(u_-)} \; ,
\end{equation}
while admissible 2-shocks ($s>0$) satisfy 
\begin{equation}\label{shockL2}
\sqrt{\sigma'(u_+)}<s<\sqrt{\sigma'(u_-)}\; . 
\end{equation}

\item[(ii)] When $\sigma''(u)$ changes sign admissible shocks of \eqref{EL} are selected by the 
Wendroff E-condition, which  for 2-shocks dictates that $\sigma$ satisfies
\begin{equation}\label{shockE}
\frac{\sigma(u) - \sigma(u_-)}{u - u_-} \ge s^2 \ge \frac{\sigma(u) - \sigma(u_+)}{u - u_+}  \ ,  
\tag{H$_{E}$}
\end{equation}
for any $u$ between $u_-$ and $u_+$.  
The Wendroff E-condition for 1-shocks reads like \eqref{shockE} with the inequalities reversed.
\end{itemize}

A discriminating criterion capturing the internal stability of shocks and used for the solution
of the Riemann problem for general fluxes is the Liu shock admissibility criterion,  \cite[Sec 8.4]{Dafermos}.
At the level of the particular system \eqref{EL} the Liu shock admissibility
criterion is equivalent to the Wendroff E-condition. The reader can easily check that the latter  (\eqref{shockE} for $s>0$)
implies at the endpoints a weak version of the Lax inequality \eqref{shockL2}, 
where strict inequality might be replaced by equality in which case the shock becomes a (right or left) contact
discontinuity.

\subsection{Traveling waves}\label{sec:travwave}

We look for a traveling wave solution  $(u^{\eps,\delta}, v^{\eps,\delta})$ of \eqref{EDD} 
in the form
\begin{equation}\label{tws}
\begin{aligned}
u^{\eps,\delta} (x,t) &= u \left ( \frac{x-st}{\sqrt{\delta}} \right)= u (\tau) \; , 
\\[5pt]
v^{\eps,\delta}(x,t) &= v \left ( \frac{x-st}{\sqrt{\delta}} \right) = v (\tau)  \, ,
\end{aligned}
\end{equation}
connecting states $(u_-,v_-)$ and  $(u_+,v_+)$ that satisfy the Rankine-Hugoniot conditions \eqref{rh}.
Setting $\tau = \frac{x-st}{\sqrt{\delta}}$ and retaining the notation $(u(\tau), v(\tau))$ for the traveling wave, 
we need to solve a system of ordinary differential equations
\begin{equation}\label{twelas}
\begin{aligned}
-s u' - v' &=0\ ,
\\
-s v' - \sigma(u)' &= \frac{\eps}{\sqrt{\delta}} v'' - u''' \, .
\end{aligned}
\end{equation}
Existence of traveling waves is a well studied problem; the objective is to provide conditions that
guarantee  convergence of the traveling wave as $\eps,\delta \to 0$ to the associated shock of \eqref{EL} in the moderate dispersion regime 
$\delta = o(\eps)$.

Consider the problem of constructing traveling wave solutions of \eqref{twelas} satisfying
$$
\lim_{\tau \to \pm \infty} u  = u_\pm, \quad  \lim_{\tau \to \pm \infty} v  = v_\pm,\quad
\lim_{\tau \to \pm \infty} v' = \lim_{\tau \to \pm \infty} u'' = 0 \, .
$$
Integrating \eqref{twelas} in $(-\infty, \tau)$ leads to solve the boundary value problem
\begin{align}
\label{eqn2}
u'' + \frac{s \eps}{\sqrt{\delta}} u' - \big ( \sigma(u) - \sigma (u_-) - s^2 (u - u_-) \big ) &= 0\ ,
\\
u(\pm \infty) = u_\pm  &\ ,
\end{align}
and define $v$ via the equation
\begin{equation}\label{determinev}
v - v_- = - s (u - u_-) \, .
\end{equation}
Necessary for the existence of traveling waves is that the end states satisfy \eqref{rh}.

In summary, denoting $c=\frac{s\eps}{\sqrt{\delta}}$,  
traveling waves are constructed by solving the ordinary differential equation 
\begin{equation}
\label{problem} 
\begin{aligned}
u'' + c u' +  \phi(u)  &= 0\ ,
\\
u(\pm \infty) &= u_\pm\ , 
\\
\phi(u) &:= - \big ( \sigma(u) - \sigma (u_-) - s^2 (u - u_-) \big )  \ .
\end{aligned}
\end{equation}
Following the approach for traveling waves of the KPP equation in \cite{Fife1979} and for the viscous Burgers-KdV equation in \cite{PR2007},
Problem \eqref{problem} is viewed as a Hamiltonian system with friction,   by setting $w = u^\prime$
and writing
\begin{equation}
\label{PS}
\begin{aligned}
\frac{du}{d\tau} &= w \; ,   
\\
 \frac{dw}{d\tau} &= - \phi(u)-cw \; .
\end{aligned}
\end{equation}
Define the potential $\Phi(u)$ by $\frac{d \Phi}{d u} = \phi (u)$ and $\Phi (u_+) = 0$, namely
\begin{equation}\label{defnPhi}
\Phi(u) =  \int^{u}_{u_+} \phi(z) ~dz = - \int^{u}_{u_+} \big ( \sigma(z) - \sigma (u_-) - s^2 (z - u_-) \big )~ dz\ .
\end{equation}
The energy of \eqref{PS},  $E(u,w)=\frac{w^2}{2}+\Phi(u)$, satisfies
\begin{equation}\label{EC}
\frac{d}{d\tau}E(u(\tau),w(\tau))=-cw(\tau)^2 \, .
\end{equation}
The associated Hamiltonian system of \eqref{PS} (when  $c = 0$) 
evolves on orbits of constant energy $\tfrac{1}{2} w^2 + \Phi(u) = E$, and its trajectories are identified 
by integrating the differential equations
\begin{equation}\label{eq:enerrel}
\frac{du}{d\tau} = \pm \sqrt{2( E - \Phi (u)) }\ .
\end{equation}

\subsection{Existence of traveling waves for $\eps$, $\delta > 0$ fixed}\label{sec:extravwave}

Existence and uniqueness (up to translation) results for traveling waves of \eqref{EDD} have been presented by several authors: Hagan-Slemrod \cite{HS1983}, Boldrini \cite{B1987}
and Bedjaoui-Lefloch \cite{BL2001,BL2002}. Traveling waves for the viscous Burgers-KdV \eqref{eq:vbKdV} lead to the same problem and were constructed in
\cite{BS1985}. These references show existence for $\eps$, $\delta > 0$ fixed and convergence to shock waves in the regime  $\delta = O(\eps^2)$. An outline 
of existence is provided in Theorem \ref{thm:existence} following  Hagan-Slemrod \cite{HS1983}.

The main theme is the study of convergence of traveling waves to shocks in the regime $o (\eps) \le  \delta < O(\eps^2)$ where traveling waves have oscillatory tails. 
We first review the framework of pertinent hypotheses
and their relation with shock-admissibility criteria, and then we prove convergence to shock waves.
We follow an approach devised in Perthame-Ryzhik \cite{PR2007} for traveling waves of scalar viscous Burgers-KdV equations \eqref{eq:vbKdV} extending their analysis to systems \eqref{EDD} with no genuine nonlinearity assumptions.

\subsubsection{\bf Hypotheses on $\sigma(u)$.}
We assume hyperbolicity $\sigma'(u) > 0$ and consider the general case that $\sigma''(u)$ changes sign.
For definiteness we restrict to  (forward moving) 2-shocks $s>0$ and states  $u_- <  u_+$. (A similar analysis can be performed for
(backward moving) 1-shocks $s < 0$ with $u_- > u_+$.)
We require the shock satisfies the Lax shock condition 
\begin{equation}\label{shockL}
\sqrt{\sigma'(u_+)}<s<\sqrt{\sigma'(u_-)}  \, ,  \qquad \qquad \qquad 
\tag{H$_L$}
\end{equation}
as well as the condition
\begin{equation}\label{shocksE}
\sigma(u) - \sigma (u_-) - s^2 (u - u_-) > 0     \quad \mbox{ for $u \in (u_- , u_+)$}\ .
\tag{H$_{sE}$}
\end{equation}
Assume also that $u_s > u_+$ is such that $\Phi (u_s) = \Phi (u_-)$ and that no root of the function $\phi(u)$ exists in
the interval $(u_+, u_s)$, that is
\begin{equation}\label{shockoE}
\sigma(u) - \sigma(u_+) - s^2 (u - u_+) < 0      \quad \mbox{ for $u \in ( u_+ , u_s)$}\ .
\tag{H$_{oE}$}
\end{equation}
The analysis we present will also apply to 1-shocks $s < 0$ with $u_+ < u_-$ by imposing the Lax shock condition \eqref{shockL1}
and reversing the inequalities in \eqref{shocksE}, \eqref{shockoE}.
\begin{figure}[hbp] \centering 
\includegraphics[width=.65\textwidth,clip,trim=0pt 0pt 0pt 0pt]{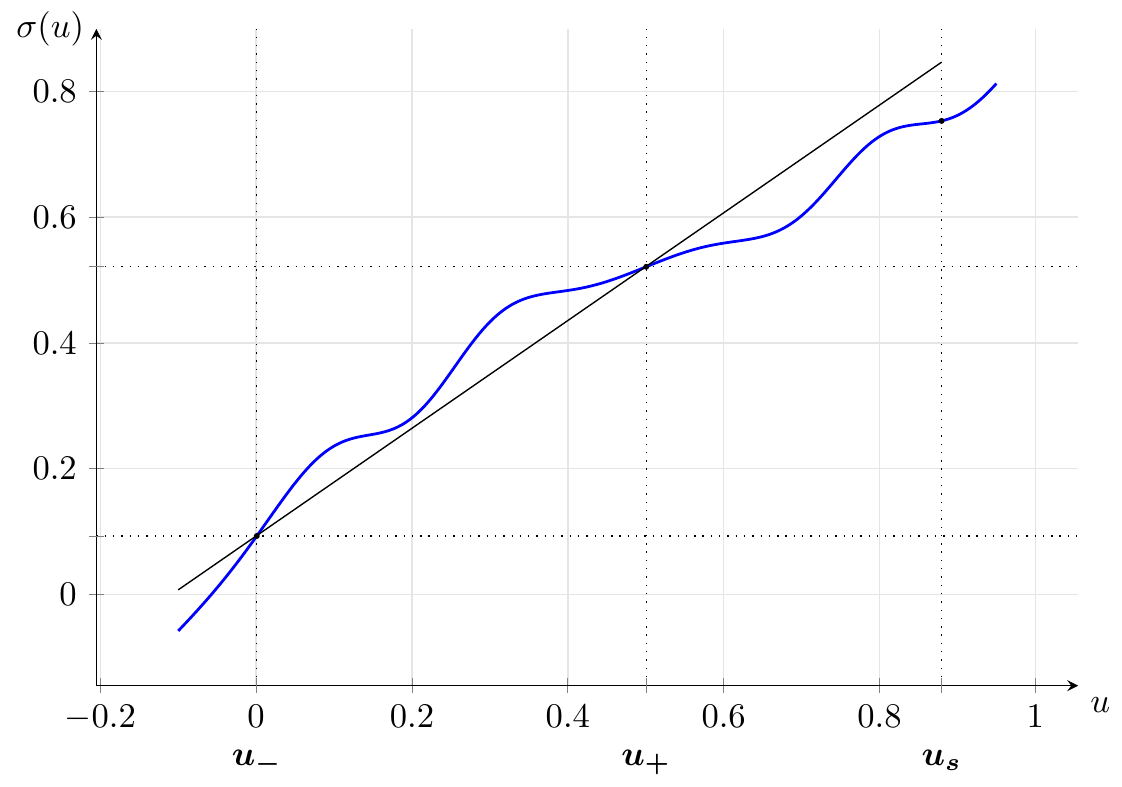}
\caption{Typical graph of $\sigma(u)$}
\label{fig:sigma}
\end{figure}
\begin{figure}[hbp] \centering 
\includegraphics[width=.65\textwidth,clip,trim=0pt 0pt 0pt 0pt]{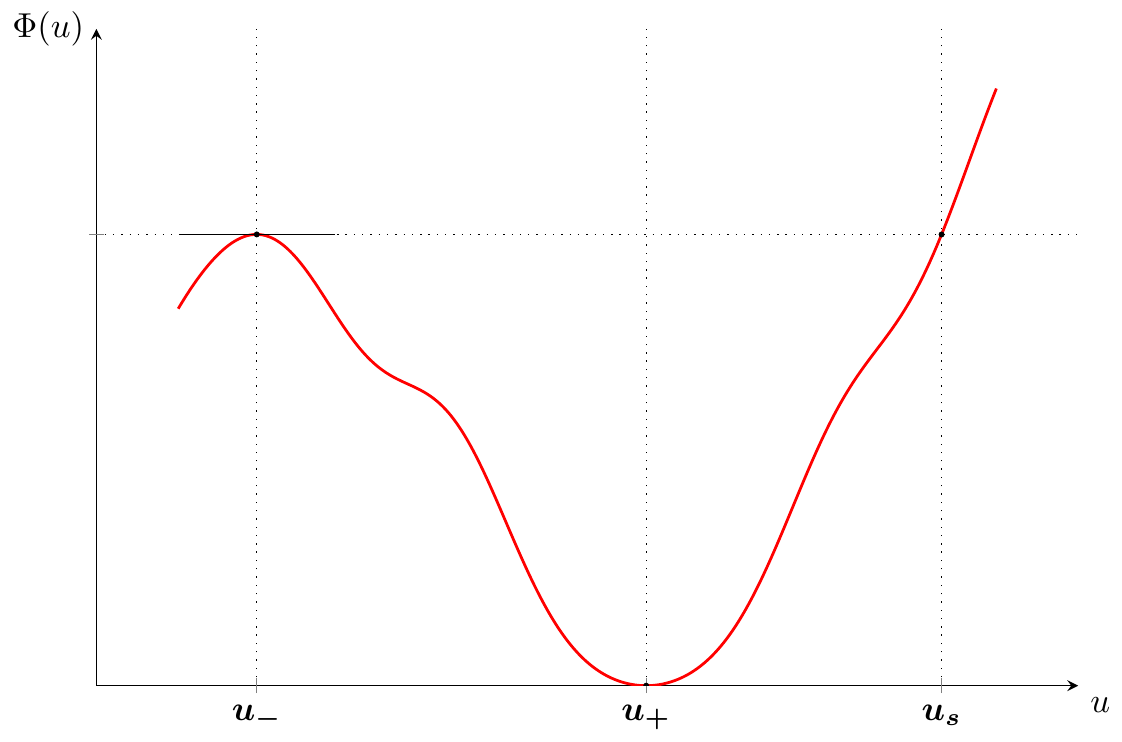}
\caption{Typical graph of $\Phi(u)$}
\label{fig:Phi}
\end{figure}

The following remarks on the hypotheses are in order: 
Condition \eqref{shocksE} is a strengthened version of the Wendroff E-condition \eqref{shockE}. Indeed, \eqref{shocksE}
implies the left inequality in \eqref{shockE} as a strict inequality and, using \eqref{shockspeed}, we obtain the right inequality 
in \eqref{shockE} again as a strict inequality. 
Hypothesis \eqref{shockL} excludes the possibility of having a contact discontinuity at the end-points $u_\pm$ 
while \eqref{shocksE} excludes  a composite shock with an internal contact discontinuity. 
(We remark that excluding contact discontinuities at the end-points can conceivably be avoided
by imposing assumptions on the order of tangency between the shock and the graph $\sigma(u)$ 
at the states $u_\pm$; we do not pursue that point here).

Hypothesis \eqref{shockoE} ensures the potential function $\Phi(u)$ has no extrema other than the critical points
in the region $[u_-, u_s]$. We refer to Figure \ref{fig:sigma} for a geometric interpretation of the 
position of the graph of $\sigma(u)$ and to Figure \ref{fig:Phi}  for the form of the potential $\Phi(u)$.

If $\sigma$ is genuinely nonlinear and for definiteness we focus on a concave $\sigma$ and a forward moving shock $s >0$,
then one easily checks that the conditions
\begin{equation}\label{HypGN}
\begin{aligned}
\sigma'' (u) &< 0 \ ,
\\
\sqrt{\sigma'(u_+)} < s &< \sqrt{\sigma'(u_-)} \ ,
\\
u_- <  u_+\,\,\, &\mbox{and} \,\,\,v_->v_+ \ ,
\end{aligned}
\tag{H$_{gn}$}
\end{equation}
imply the framework of hypotheses \eqref{shockL}, \eqref{shocksE} and \eqref{shockoE}.
(An analogous remark holds for backward moving shocks $s < 0$ with $u_- > u_+$.)


\subsubsection{\bf Existence of traveling waves.} Consider now the problem \eqref{PS}, 
where $s > 0$ satisfies \eqref{shockspeed}, $c=\frac{s\eps}{\sqrt{\delta}} > 0$, 
\begin{equation}
\label{defphi}
\begin{aligned}
\phi(u) &= - \big ( \sigma(u) - \sigma (u_-) - s^2 (u - u_-) \big )\ ,
\\
\Phi(u) &= \int_u^{u_+} \big ( \sigma(z) - \sigma (u_-) - s^2 (z - u_-) \big ) ~dz \, .
\end{aligned}
\end{equation}
We adopt the hypotheses \eqref{shockL}, \eqref{shocksE}, \eqref{shockoE} and prove the following theorem:

\begin{theorem}\label{thm:existence}
For $\eps, \delta > 0$ fixed there exists a unique up to horizontal translations traveling wave solution $(u(\tau), v(\tau))$ of the form \eqref{tws} to the system \eqref{EDD}, 
connecting the state $(u_-,v_-)$ on the left to the state $(u_+,v_+)$ on the right.
\end{theorem}

\begin{proof}
There are only two equilibrium points in the range $[u_-, u_s]$ of the system (\ref{PS}), namely $(u_-, 0)$, $(u_+, 0)$, and the linearized equations around these equilibria
become
$$
\frac{d}{d\tau} \begin{pmatrix}  U \\ W \end{pmatrix}
=  \begin{pmatrix}  0 & 1 \\  - \phi^\prime( u_\pm ) & -c \end{pmatrix} \begin{pmatrix}  U \\ W \end{pmatrix}\ .
$$
Consider the equilibrium $(u_-, 0)$. The eigenvalues are computed by
$$
\lambda^2 + c \lambda - \alpha_- = 0 \qquad \mbox{where \quad  $\alpha_- = - \phi'(u_-) = \sigma^\prime (u_-) - s^2 > 0$}\ ,
$$
they are
$$
\lambda_\pm = - \tfrac{c}{2} \pm \tfrac{1}{2} \sqrt{ c^2 + 4 \alpha_-}\ ,
$$
both real and satisfy $\lambda_- < 0 < \lambda_+$ and thus $(u_-, 0)$ is a saddle.
The directions of the stable and unstable manifolds are given by the two corresponding eigenvectors
$r_- = \left ( \begin{matrix} 1 \\ \lambda_- \end{matrix} \right )$ for the unstable and $r_+ = \begin{pmatrix} 1 \\ \lambda_+ \end{pmatrix}$ for the stable
manifold. Later we will need the property that
\begin{equation}\label{prop}
0 < \lambda_+ =: g(c) = - \tfrac{c}{2} + \tfrac{1}{2} \sqrt{ c^2 + 4 \alpha_-}  < \sqrt{\alpha_-} \ ,
\end{equation}
which is true because $g(c)$ satisfies $g'(c) < 0$, $g''(c) >0$ and thus $g(c) < g(0)$.

Next for the equilibrium $(u_+, 0)$, the eigenvalues are computed by
$$
\Lambda^2 + c \Lambda - \alpha_+ = 0 \qquad \mbox{where \quad  $\alpha_+ = - \phi'(u_+) = \sigma^\prime (u_+) - s^2  < 0$}\ ,
$$
and they are 
$$
\Lambda_\pm = - \tfrac{c}{2} \pm \tfrac{1}{2} \sqrt{ c^2 + 4 \alpha_+}\ .
$$
We distinguish two cases: (i) When $c^2 >  4  |\alpha_+|$ there are two real roots with $\Lambda_- < \Lambda_+ < 0$ and the equilibrium 
is a stable node. (ii) By contrast, in the range of weak friction $c^2 <   4  |\alpha_+|$ the eigenvalues are complex
$$
\Lambda_\pm = - \tfrac{c}{2} \pm \tfrac{1}{2} i \sqrt{ |c^2 + 4 \alpha_+ |}\ ,
$$
with negative real part and $(u_+, 0)$ is an attracting spiral.

Consider the region $\cD$ bounded by the curves
\begin{equation}\label{invdom}
w^2 = 2 \big ( \Phi(u_-) - \Phi(u) \big ) = 2 \int_{u_-}^u \sigma(z) - \sigma (u_-) - s^2 (z - u_-) \; dz\ .
\end{equation}
The curves are symmetric with respect to the $u$-axis. For $u \sim u_-$ we compute that
$$
w^2 \sim 2 \int_{u_-}^u (\sigma'(u_-) - s^2) (z - u_-) \, dz = (\sigma'(u_-) - s^2) (u- u_-)^2\ ,
$$
therefore
$$
\frac{dw}{du} (0) \sim  \pm \left(\sqrt{\sigma'(u_-) - s^2} \right) ( u - u_-)  \quad \mbox{ for $u > u_-$, $u \sim u_-$}\ .
$$

Next, in the range $u < u_s$ and $u \sim u_s$ we have
$$
\begin{aligned}
w^2 &= 2 \big ( \Phi(u_-) - \Phi (u) \big ) = 2 \big ( \Phi(u_s) - \Phi (u) \big ) 
\\
&= - 2 \int_{u}^{u_s}  \sigma(z) - \sigma (u_-) - s^2 (z - u_-) \; dz\ .
\end{aligned}
$$
Observe that by \eqref{shockoE} there exist $\alpha > 0$, $A > 0$ such that
$$
s^2 - A < \frac{\sigma(u) - \sigma(u_-)}{u- u_+} < s^2 - \alpha\ ,
$$
and thus we can show that for $u \sim u_s$, $u < u_s$  we have
$$
\alpha (u_s - u) (u + u_s - 2 u_+) < w^2 < A (u_s - u) (u + u_s - 2 u_+)\ ,
$$
which shows that the derivative $\frac{dw}{du} \sim \sqrt{u_s - u}$ in the vicinity of $u\sim u_s$.

The domain $\cD$ is enclosed by the curves $\tfrac{1}{2} w^2 + \Phi(u) = \Phi(u_-) $ which is the homoclinic orbit 
of the Hamiltonian system \eqref{PS} with $c=0$. The normal to this curve is $N = (\phi(u), w)$. Then we compute that along 
the flow of \eqref{PS} it is
$$
\left ( \frac{du}{d\tau} , \frac{dw}{d\tau} \right ) \cdot N = - c w^2 \le 0\ .
$$
The domain $\cD$ is positively invariant along the flow of \eqref{PS} and, by \eqref{prop},
the unstable manifold of the linearized system at $(u_-, 0)$ points inside $\cD$ for $c > 0$. 

The heteroclinic orbit is constructed as follows. Pick a point on the unstable manifold of \eqref{PS} at $(u_-, 0)$ which
for $u \sim u_-$ is inside $\cD$. The flow from this point backwards in time will converge to $(u_- , 0)$. Going forward in time
the flow cannot escape $\cD$ and by the Poincar\`e-Bendixon theorem it will converge to $(u_+, 0)$, giving the desired
heteroclinic orbit. This is a one-dimensional object and unique up to time-shifts.
\end{proof}

Two regimes distinguish the behavior of the heteroclinic orbit. For $c^2 >  4  |\alpha_+|$ the orbit is monotone.
For $ 0 < c^2 <  4  |\alpha_+|$ using the stable manifold theorem the orbit has an oscillatory behavior (see \cite{BS1985}).
In the following section we study the size of oscillations of the orbit.


\section{The effect of weak friction on Hamiltonian systems}\label{sec:main}
The aim of this section is to study the limit of traveling wave solutions of  \eqref{EDD} as $\eps, \delta \to 0$.
The technical vehicle is to study the oscillatory behavior for solutions $(w_c (\tau), u_c(\tau))$  to \eqref{PS}
in the regime of weak friction
\begin{equation}\label{rangec}
0 < c < c^\star  \quad \mbox{ where  $ c^\star = 2 \sqrt{|\alpha_+|} = 2 \sqrt{s^2 - \sigma^\prime (u_+)}$}\ .
\end{equation}
The $c$-dependence of solutions will be suppressed except when necessary. We prove the following theorem:

\begin{theorem}\label{thm:convergence}
 Under hypotheses \eqref{shockL}, \eqref{shocksE}, \eqref{shockoE} and for $0 < c < c^\star$ as in \eqref{rangec},
there exists a unique, up to translations, traveling wave solution $(u(\tau), v(\tau))$ to \eqref{EDD} connecting $(u_-,v_-)$ on the left to  $(u_+,v_+)$ on the right.
The solution converges strongly, as $\eps , \delta \to 0$ with $\delta = o(\eps)$ to a shock wave for \eqref{EL} that satisfies the Wendroff E-condition
(or the Liu shock admissibility criterion).
\end{theorem}

\par
The method of proof extends to the elasticity system an approach developed for scalar genuinely nonlinear equations in \cite{PR2007}.
It proceeds as follows:
\begin{itemize}
\item[(a)] We analyze the oscillatory behavior of the heteroclinic orbit  $(u_c, w_c)$ of \eqref{PS} in the range \eqref{rangec},
by estimating the energy drop in each cycle and the distance between the minima and maxima for small values of $c>0$. This is presented in Sections \ref{sec:osc1} and \ref{sec:osc2} leading to an estimate of the size of the oscillatory structure in Figure \ref{fig:PS} as a function of $c$.

\item[(b)] The information obtained in (a) is then translated at the level of traveling wave solutions
$(u^{\eps, \delta}(\tau), v^{\eps, \delta}(\tau))$ of \eqref{EDD} by means of rescaling.
\end{itemize}

The oscillatory behavior in this regime is illustrated by a numerical computation for genuinely nonlinear stress $\sigma (u) = \sqrt {u}$
and critical points $u_-=4$, $u_+=5$.  Figure \ref{fig:PS} presents the phase portrait of $(u, w)$ on the right and the form of $u(\tau)$  on the left 
for $c=0.004$, $u_-=4$, $u_+=5$.

\begin{figure}[hbp] \centering 
\includegraphics[width=\textwidth,clip,trim=0 0 0 0]{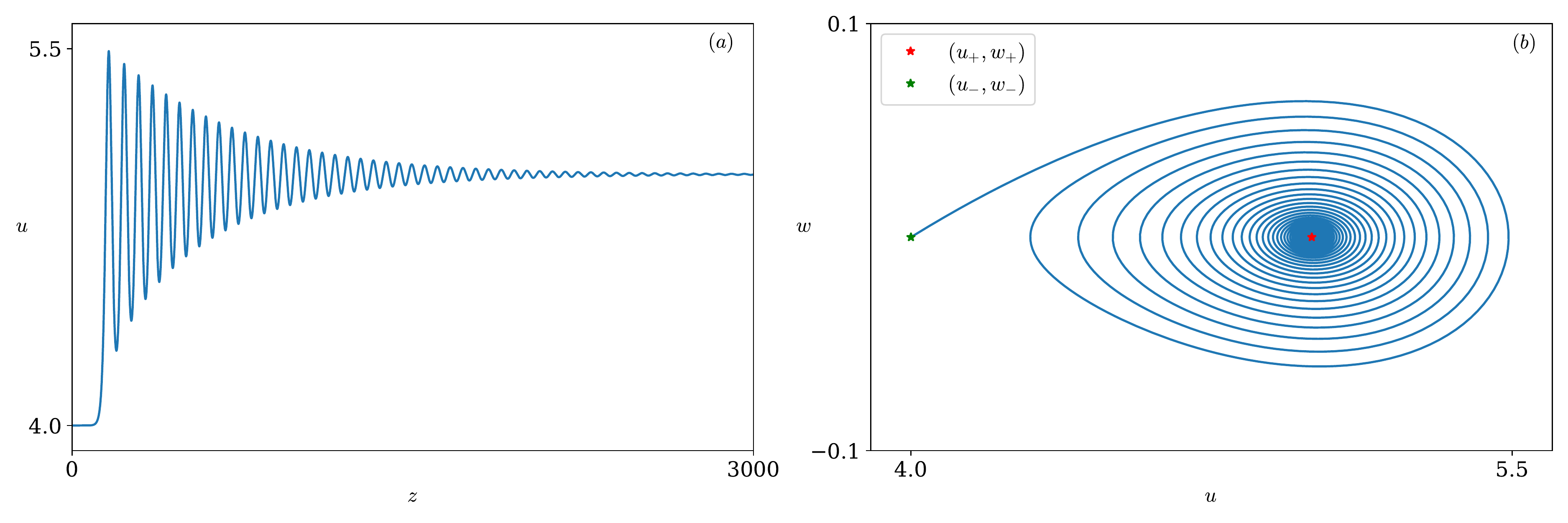}
\caption{(a) Solution $u$; (b) Phase portrait $(u, w)$. ($\sigma(u)=\sqrt{u}$, $c=0.004$, $u_-=4$, $u_+=5$) }
\label{fig:PS}
\end{figure}

Recall that for definiteness we consider the case $u_- < u_+$,  $c = s \eps / \sqrt{\delta} > 0$,  and let $u_c$ solve
\begin{equation}\label{PS2}
\begin{aligned}
u'' + c u' + \phi(u) = 0\ ,
\end{aligned}
\end{equation}
with $\phi(u) = - \big ( \sigma(u) - \sigma (u_-) - s^2 (u - u_-) \big )$, $\phi(u_-) = \phi(u_+) = 0$ and
\begin{equation}\label{defPhi3}
\Phi (u) = \int_{u_-}^u \phi (z) dz \ .
\end{equation}
Hypotheses  \eqref{shocksE}, \eqref{shockoE} imply the only extrema of $\Phi(u)$ in \eqref{defphi}  in the range $[u_-, u_s]$ 
are $u_-$ a strict local maximum and $u_+$ a strict local minimum. 
$\Phi$ is strictly decreasing on $(u_-, u_+)$, strictly increasing on $(u_+, u_s)$ and, by \eqref{shockL},
$$
\begin{aligned}
\frac{d^2 \Phi}{d u^2} &= - ( \sigma^\prime (u) - s^2) \,  , \qquad   \begin{cases}  \Phi^{\prime \prime} (u_- )  < 0 & \\ \Phi^{\prime \prime} (u_+ ) > 0 &\end{cases} \, .
\end{aligned}
$$
Select  $\alpha, \beta > 0$ such that  $\Phi(u_+ - \alpha) = \Phi( u_+ + \beta) = E_m > 0$ and
\begin{equation}\label{laxcon}
\frac{d^2 \Phi}{d u^2} (u) = s^2 -  \sigma^\prime (u)  > 0 \quad \mbox{ for \; $u \in (u_+ - \alpha , u_+ + \beta)$}\ .
\end{equation}
The range of energies is thus split to 
\begin{equation}\label{ensplit}
\begin{aligned}
&\mbox{the big energies}:   \qquad  E_m   < E < E_{max} = \Phi(u_-) \ ,
\\
&\mbox{the small energies}: \qquad   \Phi(u_+) = 0 < E < E_m\ .
 \end{aligned}
\end{equation}

Using  \eqref{shocksE}, \eqref{shockoE}, there exist $0 < \eta < H$ such that
\begin{equation}\label{estim1}
\eta ( u - u_-) <  \sigma(u) - \sigma(u_-) - s^2 (u - u_- ) < H (u - u_-), \quad  u \in ( u_- ,  u_+ - \alpha ]  \, ,
\end{equation}
and there exist $0 < m < M$ such that
\begin{equation}\label{estim2}
- M ( u - u_+) < \sigma(u) - \sigma(u_+) - s^2 (u - u_+ ) < - m (u - u_+), \quad  u \in ( u_+  , u_s ]  \, .
\end{equation}
 
\begin{remark}  In terms of the potential function $\Phi(u)$ the properties that used in the sequel
can be summarized as:
\begin{equation}\label{hypphi1}
\Phi'' ( u_-) = \phi'(u_-) < 0  \, , \quad  \Phi'' ( u_+) = \phi'(u_+) > 0 \, .
\end{equation}
For  $\alpha$, $\beta$ fixed as above, there exist $\eta, H > 0$ and $m, M > 0$ such that
\begin{equation}\label{hypphi2}
\begin{aligned}
\eta ( u - u_-)  &<   \phi(u_-) - \phi (u)  < H (u - u_-),  \, \quad   &&u \in  ( u_- , u_+ - \alpha ]\ ,
\\
m (u - u_+)    & <      \phi (u) - \phi(u_+) < M (u - u_+),  \, \quad  &&u \in ( u_+, u_s]\ .
\end{aligned}
\end{equation}
\end{remark}

\subsection{Periods of orbits of the Hamiltonian system} \label{sec:osc1}
When $c=0$, the system becomes Hamiltonian 
\begin{equation}
\label{PS0}
\begin{cases}
& \frac{du}{d\tau}=w \; ,   \\
& \frac{dw}{d\tau}=\phi(u) \; .\
\end{cases}
\end{equation}
The orbit emanating from the saddle $(u_-, 0)$ with energy $E=E_{max} = \Phi(u_-)$ is a homoclinic.
The remaining orbits  for $0 < E < E_{max}$ are periodic. Using \eqref{eq:enerrel} the period is computed by
\begin{equation}\label{per}
T(E) = 2\int_{u_1(E)}^{u_2(E)}\frac{du}{\sqrt{2(E-\Phi(u))}} \; ,
\end{equation}
where $u_1(E)$ and $u_2(E)$ satisfy $\Phi(u_1(E))=\Phi(u_2(E))=E$, see  Fig. \ref{fig:Phi}. 
First we prove the following:

\begin{lemma} \label{lem:perest}
Under hypotheses \eqref{shockL}, \eqref{shocksE}, \eqref{shockoE}, there exists $T_0 > 0$ such that
$$
T(E) \ge T_0  > 0 \, ,  \quad \forall E > 0 \, .
$$
and $T(E) \to \infty$ as $E \to E_{max}$.
\end{lemma}

\begin{proof}
We estimate the period first for {\it large energies}. The integral in \eqref{per} is split in three parts
\begin{equation}
\frac{1}{2} T(E) =  \left( \int_{u_1(E)}^{u_+ - \alpha} + \int_{u_+ - \alpha}^{u_+ + \beta}  + \int_{u_+ + \beta}^ {u_2(E)} \right) \frac{du}{\sqrt{2(E-\Phi(u))}} \ .
\end{equation}
The main contribution comes from the interval $(u_1(E) , u_+ - \alpha)$ and is estimated as follows:  From $\Phi (u_1(E)) = E$, \eqref{defphi} and  \eqref{estim1},
we have
$$
\begin{aligned}
E-\Phi(u) &= \int_{u_1(E)}^u  \big ( \sigma(z) - \sigma (u_-) - s^2 (z - u_-) \big ) ~ dz
\\
&>  \int_{u_1(E)}^u \eta (z - u_-)~ dz
\\
&= \eta (u- u_1(E) ) \left ( \frac{u + u_1(E)}{2} - u_- \right ),  \qquad u_- < u_1(E) < u < u_+ - \alpha\ ,
\end{aligned}
$$
and similarly
$$
\begin{aligned}
E-\Phi(u) &< \int_{u_1(E)}^u H (z - u_-) ~ dz
\\
&= H  (u- u_1(E) ) \left ( \frac{u + u_1(E)}{2} - u_- \right ),  \qquad u_- < u_1(E) < u < u_+  - \alpha \, .
\end{aligned}
$$
Hence, 
$$
\begin{aligned}
\frac{1}{\sqrt{H}}  F(\rho)  \le
\int_{u_1(E)}^{u_+ - \alpha}  \frac{du}{\sqrt{2(E - \Phi(u))}} \le \frac{1}{\sqrt{\eta}} F(\rho)\ ,
\end{aligned}
$$
where  $\rho = 2 (u_1(E) - u_-)$,
$$
F(\rho) := \int_0^{(u_+ - \alpha - u_1(E))}  \frac{ds}{\sqrt{s(s+ \rho )}} < \infty   \quad \mbox{for} \quad  \rho> 0 \ ,
$$
and $u_+ - \alpha - u_1(E)$ is bounded away from zero in the range of large energies.
Using the monotone convergence theorem,
$$
F(\rho) \to \infty  \quad \mbox{as \quad $\rho = 2(u_1(E) - u_-) \to 0$}\ ,
$$
that is the contribution of that integral to the period becomes infinite as $E \to E_{max}$
and the periodic orbit approaches a homoclinic orbit.

Again we consider {\it large energies} and focus on the complementary region $u \in (u_+ , u_2(E))$.
Using \eqref{defphi}, \eqref{shockspeed} and \eqref{estim2}, we deduce
\begin{align*}
E - \Phi(u) = 
\Phi (u_2(E)) - \Phi (u) &= - \int_u^{u_2(E)}  \big ( \sigma(z) - \sigma (u_+) - s^2 (z - u_+)  \big ) ~ dz
\\
&>  m  \int_u^{u_2(E)} (z - u_+) ~ dz
\\
&= \tfrac{m}{2} (u_2 (E) - u) \left( u_2(E) + u - 2 u_+ \right )\ ,
\\
\int_{u_+}^{u_2(E)}  \frac{du}{\sqrt{2(E - \Phi(u)})} &\le \frac{1}{\sqrt{m}} \int_{0}^{u_2(E)-u_+}  \frac{ds}{ \sqrt{ s  \big ( 2 (u_2 (E) - u_+) - s  \big ) }} < \infty \, ,
\end{align*}
that is the contribution of this integral to the period is finite. 
Finally, for $E > E_m$,  in the region  $u \in [u_+ -\alpha, u_+ + \beta]$,  the potential energy satisfies $0< \Phi(u) < E_m$
and the contribution of the middle integral to the period is easily seen to be bounded,
$$
\frac{\beta - \alpha}{\sqrt{2E}} < \int_{u_+ -\alpha}^{u_+ + \beta} \frac{du}{\sqrt{2(E - \Phi(u))}} \le \frac{\beta - \alpha}{\sqrt{2(E - E_m)}} \, .
$$

Next, consider the regime of {\it small energies}  $0 < E < E_m$ and analyze first the range  $u_+ - \alpha <  u_1(E) \le u \le u_+$. 
Setting $J_\alpha = [u_+ - \alpha, u_+]$
and using \eqref{laxcon} we have
$$
\begin{aligned}
\min_{J_\alpha} \sigma'(u)  &< \frac{\sigma(u_+) - \sigma(u)}{u_+ - u} < \max_{J_\alpha} \sigma'(u) \ ,
\\
- \left(\max_{J_\alpha} \Phi''(u)\right) (u_+ - u) &< \sigma(u_+) - \sigma(u) - s^2 (u_+ - u) <  - \left( \min_{J_\alpha} \Phi''(u) \right)  (u_+ - u) \ .
\end{aligned}
$$
Set $m = \min_{J_\alpha} \Phi''(u) > 0$, $M = \max_{J_\alpha} \Phi''(u) > 0$ and use
$$
E - \Phi (u) =  \int_{u_1(E)}^u \sigma(z) - \sigma(u_+) - s^2 (z-u_+) dz \ ,
$$
to obtain
$$m \int_{u_1(E)}^u (u_+ - z) ~ dz <   E - \Phi (u) < M \int_{u_1(E)}^u (u_+ - z) ~ dz\ ,$$
and hence
\begin{equation}
\label{periodsmall1}
\begin{aligned}
\frac{1}{\sqrt{M}} I(E) &\le \int_{u_1(E)}^{u_+} \frac{du}{\sqrt{ 2(E - \Phi(u))}} < \frac{1}{\sqrt{m}} I(E)\ ,
\\
\mbox{where} \qquad I(E) &= \int_{0}^{u_+ - u_1(E)} \frac{ds}{\sqrt{ s \big ( 2(u_+ - u_1(E)) - s \big ) }}\ .
\end{aligned}
\end{equation}

Again for {\it small energies}  $0 < E < E_m$ we consider next the range $u_+ \le u <   u_2 (E) < u_+  + \beta $
and similarly obtain the bound
\begin{equation}\label{periodsmall2}
\begin{aligned}
\frac{1}{\sqrt{M}} J(E) \le \int_{u_+}^{u_2(E)} \frac{du}{\sqrt{2(E-\Phi(u))}} \le \frac{1}{\sqrt{m}} J(E)\ ,
\\
\mbox{where}  \qquad J(E) = \int_0^{u_2(E) - u_+} \frac{ds}{\sqrt{s \big ( 2(u_2(E) - u_+) - s\big )}}  \ ,
\end{aligned}
\end{equation}
$m = \min_{J_\beta} \Phi''(u)$, $M = \max _{J_\beta} \Phi''(u)$, $J_\beta = [u_+, u_+ + \beta]$ and $m, M > 0$.

We conclude from \eqref{periodsmall1}, \eqref{periodsmall2} that periodic orbits with small energies have periods
of the order of $I(E) + J(E)$. The limiting behavior as $E \to 0$ is computed 
via the integral
\begin{equation}
K(a) := \int_0^a \frac{ds}{\sqrt{s (2a - s)}} = \int_0^a \frac{ds}{\sqrt{a^2 - (a-s)^2}}
= \arcsin \frac{\tau}{a} \;  \Bigg |_0^a = \frac{\pi}{2}\ .
\end{equation}
Therefore,  $J(E) = I(E) = \frac{\pi}{2}$ and $T(E)$  remains bounded from below as $E\to 0$. 
The limit of $T(E)$ can be calculated but we do not pursue this point further.
\end{proof}

\subsection{Effect of friction on orbits}\label{sec:osc2}

We study the oscillatory behavior of solutions $(u_c (\tau), w_c (\tau))$ to \eqref{PS} in the range $0 < c < c^\star$.
Let  $x_n$ be the consecutive maxima of $u_c (\tau)$ and  $y_n$ the minima. By a shift of the independent variable we set the first maximum 
at $x_0 = 0$;  this gives the ordering $y_0 = - \infty < x_0 = 0 < y_1 < x_1 < \cdots < y_n < x_n < \cdots$.

{\it Large energies.} The following lemma, for large energies, estimates the distance between two consecutive extrema, and the energy drop between these points.

\begin{lemma}
\label{LNE}
	For $c>0$ sufficiently small and energies $E_m\leq{E(y_n)}<E_{max}$, there is a constant $K>0$ such that
	$$
	 E(y_{n+1})-E(y_n)\leq-Kc  \, , \qquad 
	y_{n+1}-y_n\leq\frac{K}{(cn)^{1/2}} \ .
	$$ 
\end{lemma}

\begin{proof}
Solutions of \eqref{PS} satisfy the energy dissipation equation \eqref{EC}. 
Using the normalization $u'(0)=0$, we find that the energy drop between $-\infty$ and $\tau_0=0$ is
\begin{equation}\label{ECInf}
 \ E_{max}-E(0)=c\int_{-\infty}^{0}w^2(t)~ dt \; .
\end{equation}

Let $(u_c (\tau), w_c(\tau))$ be a solution of \eqref{PS}  and $(u_0 (\tau), w_0 (\tau))$ be a solution of the Hamiltonian system \eqref{PS0}.
Suppose that both solutions emanate from the same initial data, that is
$$
(u_c (\tau_0), w_c(\tau_0))  = (u_0 (\tau_0), w_0 (\tau_0))  = (u_0, w_0) \in \cD\ ,
$$ 
where $\cD$ is the domain in \eqref{invdom}.
Since $\cD$ is an invariant domain for both \eqref{PS} and \eqref{PS0}, using Gronwall's lemma, we obtain
\begin{equation}\label{Gr}
\begin{aligned}
|w_c (\tau)- w_0 (\tau)| + |u_c (\tau)- u_0 (\tau)| &\leq c  K_T  \sup_{|\zeta-\tau_0|\leq{T}}|w_c (\zeta)|
\\
 &\le c \, \bar K_T  \qquad \quad  \text{ for} \;  |\tau-\tau_0|\leq{T},
\end{aligned}
\end{equation}
 where $K_T$ depends on the Lipschitz constant (on the domain $\cD$) and on $T$.

We will show that
$$
\int_{-\infty}^{0}w_c^2(\tau)~d\tau\geq \kappa  > 0 \; .
$$
Indeed, without loss of generality, taking $\tau_0=0$ in \eqref{Gr}, we have $|w_c (\tau)-w_0(\tau)|\leq c \bar K$ for 
$\tau\in[-1,0]$.  Hence, 
\begin{equation}\label{w2ineq}
\int_{-\infty}^{0}w_c^2(\tau)~d\tau\geq \int_{-1}^{0}w_c^2(\tau)~d\tau \geq \int_{-1}^{0}w_0^2(\tau)~d\tau - \tilde{K} c \geq  K \ ,
\end{equation}
and
$$E(0)\leq E_{max}-Kc \; .$$

Let  now $x_n$ be a maximum and $y_n$, $y_{n+1}$ be consecutive minima of $u_c(\tau)$.
Consider two orbits  $(u_c, w_c)$ and  $(u_0 , w_0 )$ that meet at the same point in phase space, with $u_c (y_n) = u_0(y_n)$, $w_c(y_n) = w_0(y_n) = 0$,
and let $T$ be the period of the periodic orbit.
We claim that
\begin{equation}\label{converg}
 | (x_n - y_n ) - \tfrac{T}{2} | \le o(1) \quad\text{and}\quad |( y_{n+1} - y_n  ) - T | \le o(1)  \quad \mbox{as $c \to 0$} \, .
\end{equation}
This comparison between the period $T$ and the time elapsed between two consecutive extrema $x_n - y_n$ is a consequence of the fact
that the orbits $(u_c, w_c)$ converge to the orbit $(u_0, w_0)$ as $c \to 0$.  To see that observe that by \eqref{Gr},
$$
| u_c (x_n) - u_0(x_n) | \le K_{(x_n - y_n )} c  \, , \quad | u_c (y_n + \tfrac{T}{2} ) - u_0(y_n + \tfrac{T}{2} ) | \le K_T c \, , 
$$
and then use \eqref{per} and Lemma \ref{lem:perest} to obtain
$$
\begin{aligned}
(x_n - y_n) - \tfrac{T}{2} &= \int_{u_0 (y_n)}^{u_0(x_n)}  \frac{du}{\sqrt{2(E-\Phi(u))}} - \int_{u_0 (y_n)}^{u_0 \big ( y_n + \tfrac{T}{2} \big)}  \frac{du}{\sqrt{2(E-\Phi(u))}}
\\
&= \int_{u_0 \big ( y_n + \tfrac{T}{2} \big)}^{u_c (x_n) + O(c)}   \frac{du}{\sqrt{2(E-\Phi(u))}} \to 0  \qquad \mbox{as $ c \to 0$}\ .
\end{aligned}
$$
Similarly is proved the second identity in \eqref{converg}.

This shows that for $E_m\leq{E(y_n)}<E_{max}$ we have $y_{n+1} - y_n \geq T_0$ where $T_0$ does not depend on $n$, and thus
\begin{equation}\label{decay0}
E(y_{n+1} ) -  E( y_n )  \le - c\int_{y_n}^{y_{n+1}}w_c^2(\tau)~d\tau \le - K c\ ,
\end{equation}
and for the energy at $y_n$,
\begin{equation}\label{decay}
E(y_n)\leq E_{max}-Kcn \; .
\end{equation}

For  $(u(\tau), w(\tau))$ solution of \eqref{PS}, we
proceed to estimate the distance between  two consecutive minimum at $y_n$ and maximum at $x_n$. 
The domain $[y_n , x_n]$ is split to three regions:
\begin{itemize}
\item[(I)]  $u_- <  u(y_n) < u (a_n ) = u_+ - \alpha$ for $y_n \le \tau \le a_n$
\item[(II)]  $u_+ - \alpha = u(a_n) < u_+ < u(b_n) = u_+ + \beta$ for $a_n \le \tau \le b_n$
\item[(III)] $u_+ + \beta = u(b_n) < u(x_n) < u_s $ for $b_n \le \tau \le x_n$
\end{itemize}

In Region (II) we have
$$
0 < u(b_n) - u(a_n) = w(\tau_\star) (b_n - a_n ) \qquad \mbox{for some $\tau_\star \in [a_n , b_n ]$}\ .
$$
Since the orbit of $(u(\tau), w(\tau)$ is near the orbit $(u_0(\tau), w_0(\tau))$ we have $w(\tau) \ge \kappa > 0$ and we conclude
$b_n - a_n \le K$.

Consider next the Region (I) and observe that using \eqref{decay}, \eqref{defphi}, \eqref{estim1},
$$
\begin{aligned}
K n c &\le \Phi (u_-) - \Phi (u(y_n))
\\
&= ( - \phi (v) ) ( u(y_n) - u_- ) \qquad \mbox{for some $u_- < v < u_+ - \alpha$}
\\
&\le  \left ( \max_{u_- < v < u_+ - \alpha} (- \phi (v) ) \right ) ( u(y_n) - u_- ) 
\\
&\le \max_{u_- < v < u_+ - \alpha} \big (  H (v - u_-) \big ) ( u(y_n) - u_- ) 
\\
&\le K' ( u(y_n) - u_- ) 
\end{aligned}
$$
Using \eqref{problem} and \eqref{estim1},
$$
u'' + c u' = - \phi(u) > \eta (u- u_-) > \eta (u(y_n) - u_-) > K n c
$$
whence
$$
\big ( u(\tau) - u_- \big )^\prime \ge K n \big ( 1 - e^{-c (\tau - y_n )}  \big ) 
$$
and integrating once again and using $e^{-cx} \ge 1 - cx + \tfrac{1}{2} c^2 x^2 - \tfrac{1}{6} c^3 x^3$ we derive
$$
\begin{aligned}
u(\tau) - u(y_n)  &\ge  \frac{Kn}{c} \Big ( c(\tau - y_n) + e^{-c (\tau - y_n)} -1 \Big )
\\
&\ge K n c \tfrac{1}{2} (\tau - y_n)^2 \big ( 1 - \frac{1}{3} c (\tau - y_n ) \big )
\\
&\ge K n c \tfrac{1}{4} (\tau - y_n)^2   \qquad \mbox{provided $c (x_n - y_n) < \tfrac{3}{2}$}
\end{aligned}
$$
We conclude
$$
\tau - y_n \le \frac{K}{\sqrt{nc}}  \qquad \mbox{for $y_n \le \tau \le a_n$}\ .
$$

On the Region (III)  $b_n \le \tau \le x_n$ the estimate proceeds along similar lines: First using \eqref{estim2},
$$
K n c \le \Phi (u_s) - \Phi (u (x_n)) \le M (u_s - u(x_n))\ .
$$
Using \eqref{problem} and \eqref{estim2}
$$
(u_s - u)'' + c (u_s - u)' = \phi (u) \ge \min_{u \in [u_+ + \beta, u_s]} \phi (u) =: A > 0\ ,
$$
and after an integration
$$
u' (\tau) \ge \frac{K}{c} \big ( e^{c(x_n - \tau)} - 1 \big ) \qquad b_n \le \tau \le x_n\ ,
$$
and another one
$$
\begin{aligned}
u(x_n) - u(\tau) &\ge \frac{K}{c}  \int_\tau^{x_n}  \big ( e^{ c(x_n - z)} - 1 \big ) dz
\\
&\ge \frac{K}{2} (x_n - \tau)^2  \qquad \mbox{for $b_n \le \tau \le x_n$}\ .
\end{aligned}
$$
We conclude that
$$
x_n - \tau \le  \Big ( \frac{2(u(x_n) - u(\tau))}{K} \Big )^{\tfrac{1}{2}} \le K^\prime\ .
$$
Putting all together we obtain
\begin{equation}
x_n - y_n \le \frac{K}{ (nc)^{1/2}} \ ,
\end{equation}
and similarly by a symmetric argument  $y_{n+1} - x_n \le \frac{K}{ (nc)^{1/2}} $ which completes the proof.
\end{proof}

At this point we estimate the "length" of the highly oscillatory part of the solution. 
Since the energy drop per period is  of size $Kc$ the total number of oscillations is $N=K/c$. The length of this regiion is
\begin{equation}\label{length-scale}
L = \sum_{n=1}^{N=K/c} y_{n+1} - y_n 
\le \sum_{n=1}^{N=K/c} \frac{K}{(cn)^{1/2}} 
=\frac{K}{c^{1/2}} \int_1^{\tfrac{K}{c}} \frac{dx}{x^{1/2}} \le \frac{K'}{c}\ .
\end{equation}

{\it Small Energies.} For energies  $E(\tau)\in(0,E_m)$, we show the exponential damping behavior of the solution. Since $\Phi''(u) > 0$ in this region,
 the situation is analogous to the analysis in \cite{PR2007}.  Using the energy dissipation  \eqref{EC}, we obtain
$$E'(\tau)\geq -cE(\tau) \ ,
$$
and thus 
$$E(\tau)\geq E(\tau_0)e^{-c(\tau-\tau_0)} \ .$$

To show the opposite inequality note that in this region $\sup|w|  \le \sqrt{2 E(\tau_0)}$. Then Gronwall's inequality \eqref{Gr}
implies
\begin{equation}\label{GrEx}
|u_c (\tau)-u_0(\tau)|+|w_c(\tau)-w_0(\tau)|\leq c \sqrt{E(\tau_0)}e^{L(\tau-\tau_0)} \, .
\end{equation}
From the analysis of \eqref{PS0} in section \ref{sec:osc1} for small energies, the distance
between consecutive maxima satisfies
$x_n-x_{n-1} \ge \alpha > 0$ for some $\alpha > 0$ independent of $n$, and same for the minima $y_n$. 
Following analogous to the large energies case steps in \eqref{decay0}, we obtain an upper bound for the energy
$$
E(x_n)\leq E(x_{n-1}) \big(1-Kc (x_n - x_{n-1}) \big ) \; .
$$
We conclude that for $E < E_m$ the energy decays exponentially at a rate $cK$, 
$$
E(x_n) \le E(x_0) e^{-K c (x_n - x_0)} \, ,
$$
where $x_0$ is the first point of maximum such that  $E(x_0) \le E_m$. It follows that 
$u_c (\tau) \to u_+$ on a length scale of order $1/c$.

\medskip
The result that has been proved can be expressed entirely based on the second order equation \eqref{PS2} and properties  of the function $\phi(u)$.
We summarize the result in a  proposition.

\begin{proposition}\label{propsize}
Let $u_c (\tau)$ be heteroclinic connections of \eqref{PS2} with $u_c (\pm \infty) = u_\pm$ with  $0 < c < c^{\star}$,
where $\Phi(u)$ is defined in \eqref{defPhi3}, 

(i) Let $u_- < u_+$ and assume  $\phi(u)$ satisfies
\begin{equation}\label{hyppho0}
\begin{aligned}
\mbox{ $u_-$, $u_+$ are the only solutions of $\phi(u) = 0$  in $[u_-, u_s]$}
\\
0 = \Phi(u_+) < \Phi(u) < \Phi (u_-) = \Phi(u_s) = E_{max} \quad \text{for}\quad u \in (u_- , u_s)\ ,
\end{aligned}
\tag{H$_{\phi 0}$}
\end{equation}
as well as
\begin{align}
\label{hphi1}
&\phi(u) < 0  \, , \qquad u_- < u < u_+\ ,
\tag{H$_{\phi 1}$}
\\
\label{hphi2}
&\phi(u) > 0   \, , \qquad u_+ < u \le u_s\ ,
\tag{H$_{\phi 2}$}
\\
\label{hphi3}
&\phi^\prime (u_-) < 0 \, , \quad \phi^\prime (u_+) > 0\ .
\tag{H$_{\phi 3}$}
\end{align}
Let $u_c(\tau)$ be normalized by setting $u^\prime_c (0) = 0$.  Then the domain $\tau > 0$ is split into two regions:
\begin{itemize}
\item the region where the solution $u_c(\tau)$ has large amplitude oscillations with energy  $E_m < E < E_{max}$  and which has length of size $O (\tfrac{1}{c})$.
\item the region where the solution $u_c(\tau)$ has small amplitude oscilations of energy $0 < E < E_m $ which has again length of size  $O (\tfrac{1}{c})$.
\end{itemize}
\end{proposition}

One can easily check that \eqref{hphi1}, \eqref{hphi2}, \eqref{hphi3} imply that  and $\phi(u)$ satisfies \eqref{hypphi1} and \eqref{hypphi2} 
with $\alpha, \beta$ as defined in \eqref{laxcon}, which are the key ingredients on which the analysis of section \ref{sec:main} is based.

An analogous result can be proved for the case $u_- > u_+$  and $c > 0$ under the hypotheses that $\Phi$ has a maximum at $u_-$,
a (nondegenerate) minimum at $u_+$ and $\phi(u)$ satisfies conditions analogous to  \eqref{hypphi1}, \eqref{hypphi2} in the rest
of the domain $[u_s , u_-]$ where $\Phi(u_s) = \Phi(u_-)$.  These two results provide, by performing a change of direction $\tau \to - \tau$,
two analogous results valid for the case that the shock speed $s < 0$.

\subsection{Returning to the original variables}\label{sec:return}

Consider now the convergence of traveling waves for \eqref{EDDI} as $\eps, \delta \to 0$.
Recall the traveling wave is expressed via
\begin{equation}
\begin{aligned}
u^{\eps,\delta} (x,t) &= U^{\eps, \delta} \left ( x-st \right) = u_c \left ( \frac{x-st}{\sqrt{\delta}} \right) \; , 
\\[5pt]
v^{\eps,\delta} (x,t) &= V^{\eps, \delta} \left ( x-st \right) = v_c \left ( \frac{x-st}{\sqrt{\delta}} \right ) \, ,
\end{aligned}
\end{equation}
where $u_c$ solves \eqref{problem},  $v_c$ is determined by \eqref{determinev} and $c = s \tfrac{\eps}{\sqrt{\delta}}$ (and here we consider the case that $s>0$).

We fix the shift of the traveling wave so that $\theta = 0$ at the first maximum of the function $u_c (\tau)$. Since $U^{\eps, \delta}$ and $u_c$ are related through
the scaling
\begin{equation}
U^{\eps, \delta} (\theta) = u_c \left ( \frac{\theta}{\sqrt{\delta}} \right )\ ,
\end{equation}
the graph of $U^{\eps, \delta}$ is obtained from the graph of $u_c$ by scaling down in the axis direction by a factor $\sqrt{\delta}$.
Accordingly, a structure of length scale $L$ in the graph of $u_c$ will have length scale $\sqrt{\delta} L$ in the graph of $U^{\eps, \delta}$.

We have seen that $u_c (\tau) \to u_\pm$ as $\tau \to \pm \infty$, and from the analysis leading to \eqref{length-scale} the
region of oscillations at the high energy regime is  of order $\frac{1}{c}$. In the regime of small energies the length scale
of oscillations is again of order $\frac{1}{c}$.  When we transfer these length scales at the level of the original variables, they both become
of size $O(\sqrt{\delta} \frac{1}{c} )= O ( \delta / \eps)$ and they will shrink to zero provided that $\delta = o(\eps)$.
This finishes the proof of the main theorem. 

\subsection{Is $\delta= o(\eps)$ optimal ? }\label{sec:linsystem}

The range $\delta = o(\eps)$ is optimal for the linearized system associated with  \eqref{problem}. The linearized
equation around $u_+$ takes the form
$$
\frac{d^2 \tilde u}{d\tau^2} + c \frac{d\tilde u}{d\tau} + \alpha \tilde u = 0\ ,
$$
with $c = \tfrac{s \eps}{\sqrt{\delta}} $,  $\alpha = \phi (u_+) > 0$ and $\tilde u = u - u_+$. The characteristic polynomial for the linear differential equation
has  complex eigenvalues when $c \ll 1$ which are 
$\rho_{\pm} = - \frac{c}{2} \pm \frac{i}{2} \sqrt{ |4 \alpha - c^2|}$, 
and its solution is 
$$
\tilde u (\tau) =  A e^{ -\tfrac{c}{2} \tau } \cos ( \omega \tau + \beta)\ ,
$$
where $A$ is an amplitude, $\omega = \sqrt{\alpha - \tfrac{c^2}{4}} $ the frequency and $\beta$ a phase shift.
When expressing the solution in terms of the original variables we have
$$
 u^{\eps, \delta} \left ( \frac{x- s t}{\sqrt{\delta}} \right ) - u_+  =  A  e^{-  \tfrac{\eps}{ \delta} \tfrac{s}{2} (x - s t)} \cos \left (  \left( \sqrt{\alpha - \tfrac{s^2 \eps^2}{4\delta}} \right )
 \frac{x-st}{\sqrt{\delta}} + \beta \right )\ .
$$
In the regime of moderate dispersion  $\delta = o(\eps)$ the right side converges to zero as $\delta, \eps \to 0$ and the traveling wave converges to a shock.

In the moderate dispersion regime $\delta = o(\eps)$, the convergence to a shock wave is in a strong sense as often expected in shock wave theory.
In the complementary region  $\eps = o(\delta)$ the solution of the linearized equation oscillates vigorously around the constant state $u_+$.
Such an oscillatory tail converges to a constant state in a weak sense since the average of the oscillations around the constant $u_+$
cancel out. 
This behavior characterizes  only the linearized problem and it is not clear if it will persist for the nonlinear problem.
Weak convergence could conceivably give a meaning on how the limiting state $u_+$ is achieved even in a regime of strong dispersion. 


\section{Dispersive Shocks in Quantum Hydrodynamics with Viscosity}
\label{sec:qhd}

We consider the one dimensional quantum hydrodynamics system (QHD) with artificial viscosity
\begin{equation}
\label{QHDDD}
\begin{aligned}
\rho_t+j_x &=0\ ,  
 \\
 j_t+\left(\frac{j^2}{\rho}+\rho^\gamma\right)_x &= \eps j_{xx}+\delta\rho \left(\frac{\sqrt{\rho}_{xx}}{\sqrt{\rho}}\right)_x \ ,
\end{aligned}
\end{equation}
where $\rho$ is the density, $u$ the fluid velocity, $p(\rho)$ the pressure, and $j$ the fluid momentum, $j=\rho u$.
The Bohm's potential, also known as Quantum potential, represents a  dispersive term and artificial viscosity is also introduced in this system 
modeling effects of dissipation.  Our objective is to show that the combined effect of diffusion and dispersion leads in the moderate dispersion
regime to a dispersive shock wave with oscillatory tails.
Traveling wave solutions of \eqref{QHDDD} have been studied in \cite{LMZ2020} in a weak dispersion regime  $\delta=\eps^2$, 
where  existence of traveling waves and convergence to a shock when $\delta=\eps^2 \to 0$ is proved.

At first sight, the nonlinearities and dispersion in the system \eqref{QHDDD} appear  more complex than in the system \eqref{EDD}, but casting the problem in the right variables will
 transform the traveling wave analysis to examining a Hamiltonian system with weak friction   \eqref{PS}.

\subsection{Shocks in gas dynamics}\label{sec:shockgd}

 The system of isentropic gas dynamics 
\begin{equation}
\label{HYP}
\begin{aligned}
\rho_t+ (\rho u)_x &= 0\ ,   
\\
 (\rho u)_t+( \rho u^2 + p(\rho) )_x &= 0\ ,
\end{aligned}
\end{equation}
is hyperbolic when $p'(\rho) > 0$ and has eigenvalues $\lambda_\pm = u \pm \sqrt{ p^\prime (\rho)}$ and corresponding
right eigenvectors $r_\pm = \big ( \pm \rho, \sqrt{ p^\prime (\rho)} \big )^T$. Under the condition 
\begin{equation}
\label{gneuler}
\big ( \rho^2 p^\prime(\rho) \big )^\prime  > 0\ ,
\end{equation}
both characteristic fields are genuinely nonlinear $r_\pm \cdot \nabla \lambda_\pm > 0 $.

Shock waves are discontinuous solutions of \eqref{HYP} connecting two states $(\rho_- , u_- )$ to $(\rho_+ , u_+)$.
They have been studied extensively, {\it e.g.} \cite[Ch 18]{Smoller}. Shocks are constructed by solving the Rankine-Hugoniot conditions
\begin{equation}\label{RH}
\begin{aligned}
- s [\rho]  + [\rho u] &= 0 
\\
- s   [\rho u] + [\rho u^2 + p] &= 0
\end{aligned}
\end{equation}
where $s$ is the shock speed, and we use the usual notation $[\rho] = \rho_+ - \rho_-$ etc. Introduce $m = \rho u - s u$ and note that
$$
[m] = 0 \quad \text{and}\quad m [u] + [p] = 0\ .
$$
This implies that $m$ stays constant across the shock
\begin{equation}\label{massflux}
\rho_+ ( u_+ - s) = \rho_- ( u_- -s) = :m \, , \quad  m = - \frac{p_+ - p_-}{u_+ - u_-} \ , 
\end{equation}
where $p_+ = p(\rho_+)$, $p_- = p(\rho_-)$ and $m$ is computed by 
\begin{equation}\label{eqm2}
m^2 = - \frac{ p_+  - p_- }{ \frac{1}{\rho_+} - \frac{1}{\rho_-}} \ .
\end{equation}

A 1-shock associated to the $\lambda_-$ characteristic speed satisfies the Lax shock condition when
$$
u_+ - \sqrt{p'(\rho_+)} < s < u_- - \sqrt{p'(\rho_+)} \ .
$$
If \eqref{gneuler} is satisfied using \eqref{massflux} one checks
$$
\sqrt{ \rho_+^2 p'(\rho_+)} > m > \sqrt{ \rho_-^2 p'(\rho_-)} > 0\ ,
$$
and we deduce that for a Lax admissible 1-shock we have
$$
\rho_+ > \rho_- \, , \quad m > 0 \, , \quad s < u_+ < u_- \ .
$$

A 2-shock associated to the $\lambda_+$ characteristic speed satisfies the Lax shock condition when 
$$
u_+ + \sqrt{p'(\rho_+)} < s < u_- + \sqrt{p'(\rho_+)} \ .
$$
In a similar way, using \eqref{gneuler}, we deduce
$$
\rho_+  <  \rho_- \, , \quad m <  0 \, , \quad  u_+ < u_- < s \ .
$$

The systems \eqref{EL} and \eqref{HYP} are equivalent 
by the transformation $\hat y (\cdot, t) : x \to y$ from Lagrangian to Eulerian coordinates, \cite{Smoller}. 
Indeed, using $y = \hat y(x,t)$, $v = \frac{\del \hat y}{\del t}$ and $ w = \frac{\del \hat y}{\del x}$, the equations
$$
w_t = v_x  \, , \quad  v_t = \sigma(w)_x\ ,
$$
express respectively the equality of mixed partial derivatives and the balance of momentum in Lagrangian coordinates. The balance of mass takes the form
$\rho (\hat y(x,t), t) = \frac{1}{w}(x,t)$ and may be viewed as defining the density. The velocity in Eulerian coordinates $u (y,t)$ relates to the Lagrangian velocity $v(x,t)$ 
through $v (x,t) = \hat u (\hat y (x,t), t)$. Using these formulas one 
checks that $(\rho ,  u )(y,t)$ satisfy the system \eqref{HYP} with the identification for the pressure
\begin{equation}\label{chvar}
p(\rho) = - \sigma \left ( \frac{1}{\rho} \right )  \, , \quad w = \frac{1}{\rho}\ .
\end{equation}
The Rankine-Hugoniot conditions, the equations for the shock curves, and the Lax shock-admissibility conditions in Lagrangian and Eulerian coordinates transform to each other.
In particular, the condition \eqref{gneuler} in Eulerian coordinates corresponds to the condition $\sigma^{\prime \prime} (w) < 0$ 
for the Lagrangian counterpart.

\subsection{Reduction of traveling waves  to a Hamiltonian system with friction}

Consider \eqref{QHDDD} and introduce for its solution $( \rho^{\eps, \delta} , j^{\eps, \delta})$  the {\em ansatz} of traveling waves,
$$
\begin{aligned}
\rho^{\eps, \delta}  (x,t)= \rho  \left ( (x-st)/\sqrt{\delta} \right ) = \rho (\tau) \, , 
\\
j^{\eps, \delta} (x,t)= j \left (  (x-st)/\sqrt{\delta} \right )=  j (\tau)  \, , 
\end{aligned}
$$
where $j = \rho u$, $\tau = (x-st)/\sqrt{\delta}$ and (with a slight abuse of notation) we retain the notation $(\rho (\tau), j (\tau))$ for the solution
of the traveling wave equations
\begin{equation}\label{twsystem}
\begin{aligned}
- s \rho^\prime + j^\prime &= 0 \ ,
\\
-s j^\prime + \left( \frac{j^2}{\rho} + p(\rho) \right)^\prime &= \frac{\eps}{\sqrt{\delta}} j^{\prime \prime} + \rho \left (  \frac{  (\sqrt{ \rho})^{\prime\prime}}{\sqrt{\rho}} \right )^\prime\ .
\end{aligned}
\end{equation}

Next, we fix $(\rho_-,  u_-)$, $(\rho_+, u_+)$ and $s$ that satisfy the Rankine-Hugoniot conditions \eqref{RH}. The first equation in \eqref{twsystem} gives 
\begin{equation}\label{massflux1}
\rho_- (u_- - s) = \rho (u - s) = m\ ,
\end{equation}
where the constant mass flux (relative to the shock) $m$  is computed via \eqref{eqm2}.
Using the well known formula
$$
\rho \left(\frac{\sqrt{\rho}_{xx}}{\sqrt{\rho}} \right)_x=\frac{1}{2} \left( \rho(ln\rho)_{xx} \right)_x\ ,
$$
for the Bohm potential, we integrate \eqref{twsystem} and obtain
\begin{equation}\label{twqhd}
\begin{aligned}
-s (\rho - \rho_-) + (\rho u - \rho_- u_- ) &= 0\ ,
\\
-s^2 \rho^\prime + \big ( \rho u^2 - \rho_- u_-^2 + p(\rho - p(\rho_-) \big ) &= \frac{s \eps}{\sqrt{\delta}} \rho^{\prime\prime} +  \frac{1}{2} \left ( \rho(ln\rho)^{\prime \prime} \right )^\prime
\ .
\end{aligned}
\end{equation}
In turn, setting $c = \frac{s \eps}{\sqrt{\delta}}$ and using \eqref{RH}--\eqref{eqm2} we arrive at
\begin{equation}\label{basic}
\frac{1}{2} \rho \left(  \ln \rho \right)^{\prime\prime} + c \rho^\prime - \left( p(\rho )- p(\rho_-) + m^2 \left ( \frac{1}{\rho} - \frac{1}{\rho_-} \right ) \right) = 0\ .
\end{equation}

Setting $x = \ln \rho$ we see that $x(\tau)$ solves the equation
\begin{equation}\label{secondorder}
x'' + 2 c x' + \psi(x) = 0\ ,
\end{equation}
where the function $\psi(x)$ can be expressed in the following equivalent forms
\begin{align}
\psi(x) &= -  \frac{2}{\rho} \left( p(\rho )- p(\rho_-) + m^2 \left ( \frac{1}{\rho} - \frac{1}{\rho_-} \right ) \right) \Bigg |_{\rho = e^x}
\nonumber
\\
&= 2 w \left ( \sigma (w) - \sigma (w_-) - m^2 (w - w_-) \right ) \bigg |_{w = e^{-x}}\ ,
\label{formpsi}
\end{align}
where we used the changes of variables $\rho = e^x$ for the first identity, and the formula \eqref{chvar} and $w_- =  \frac{1}{\rho_-}$
for the second. This allows to write $\psi(x)$ in the form $\psi(x) = g (e^x) = f(e^{-x})$ where
$$
f(w) = 2 w \big ( \sigma (w) - \sigma (w_-) - m^2 (w - w_-) \big )\ .
$$
The potential function $\Psi(x)$ is defined up to an arbitrary constant via
$$
\begin{aligned}
& \Psi(x) := -2 \int_{w_-}^w \big ( \sigma (z) - \sigma (w_-) - m^2 (z - w_-) \big ) \, dz \Bigg |_{w = e^{-x}} + \text{\em Const.}
\\
& \frac{d \Psi}{dx} (x) = f(e^{-x}) = \psi (x)\ .
\end{aligned}
$$

\subsection{Convergence from oscillating traveling waves to shocks} 
We established that the traveling wave problem \eqref{twqhd} reduces to solving \eqref{secondorder} for $x(\tau)$
and then defining
$$
\rho(\tau) = e^{x(\tau)} \, , \quad u(\tau) = s +  \frac{m }{\rho(\tau)}
$$
The aim is to apply Proposition \ref{propsize} to \eqref{secondorder}. The hypotheses can in principle be checked for the following reasons:
The genuine nonlinearity hypothesis for $\sigma(w)$ suggests that $f(w)$ has a sign between the roots $f(w_-) =  f(w_+) = 0$. Moreover,
$$
\frac{df}{dw} (w) = 2 \big ( \sigma (w) - \sigma (w_-) - m^2 (w - w_-) \big ) + 2w ( \sigma^\prime (w) - m^2)
$$
$$
\frac{df}{dw} (w_\pm ) = 2w_\pm ( \sigma^\prime (w_\pm) - m^2)  \, , 
$$
and since $w > 0$ the sign of $\frac{df}{dw} (w_\pm )$ amounts to  the Lax shock conditions.

We present the details for a 2-shock that satisfies the Lax conditions. 
Then $\rho_+  <  \rho_- $, $m <  0$ and $u_+ < u_- < s$. The system \eqref{QHDDD} is invariant under the change of variables
$$
\hat x = x - \kappa t \, \quad \hat u = u + \kappa \, , \quad \hat \rho = \rho  \, , \quad \mbox{for $\kappa \in \R$}.
$$
Therefore by a change of variables we may assume $u_+ = 0 < u_- < s$. This amounts to observing the flow from a coordinate system 
moving with the velocity of the fluid at $\infty$. The values $\rho_+ < \rho_-$, $m < 0$ remain unchanged.

Next, we employ the change of variables  $x = \ln \rho$ and proceed to verify  the hypotheses of Proposition \ref{propsize} 
for the arrangement $x_s < x_+ < x_-$. Then $x(\tau)$ satisfies \eqref{secondorder} with $\psi(x)$ given by \eqref{formpsi}. 
Note that $w_- = \tfrac{1}{\rho_-} < \tfrac{1}{\rho_+} = w_+$. Since $\sigma''(w) < 0$ we have 
$$
f(w) = 2 w \big ( \sigma (w) - \sigma (w_-) - m^2 (w - w_-) \big ) \, , \quad w \in (w_- , w_+)\ ,
$$
and $f(w) > 0$ on $(0, \infty) - (w_- , w_+)$. Moreover, 
$$
\frac{df}{dw} (w_-) = 2w_- ( \sigma^\prime (w_-) - m^2) > 0 \quad\text{and}\quad \frac{df}{dw} (w_+ ) = 2w_+ ( \sigma^\prime (w_+) - m^2) < 0 \, .
$$
All hypotheses of Proposition \ref{propsize}  are fulfilled for the arrangement $x_s  < x_+ < x_-$ with a maximum at $x_-$
and minimum at $x_+$.  Proceeding as in section \ref{sec:return} we have:

\begin{theorem}
\label{TQHD}
Let $p(\rho)$ satisfy \eqref{gneuler} and suppose that $s$, $(\rho_-, u_-)$, $(\rho_+, u_+)$ define a 1-shock (or a 2-shock) that satisfies the Lax shock conditions.
There exist a unique (up to shifts) traveling wave solution $(\rho^{\eps,\delta}, (\rho u)^{\eps,\delta} )(x-st)$ to the system \eqref{QHDDD} that connects state $(\rho_-, u_-)$  to 
$(\rho_+,u_+)$. When the shift is appropriately selected,  the traveling wave converges strongly as $\eps$, $\delta \to 0$ with $\delta = o(\eps)$ 
to the Lax-shock solution of \eqref{HYP}. 
\end{theorem}

As an illustration, we present in Figure \ref{fig:PSqhd} a numerical solution to \eqref{basic} for $p(\rho) = \rho^\gamma$
with $\gamma=1.4$, between states $\rho_-=1.5$, $\rho_+=1$, with the parameter $c=0.02$. 

\begin{figure}[ht!] \centering 
\includegraphics[width=\textwidth,clip,trim=0pt 0pt 0pt 0pt]{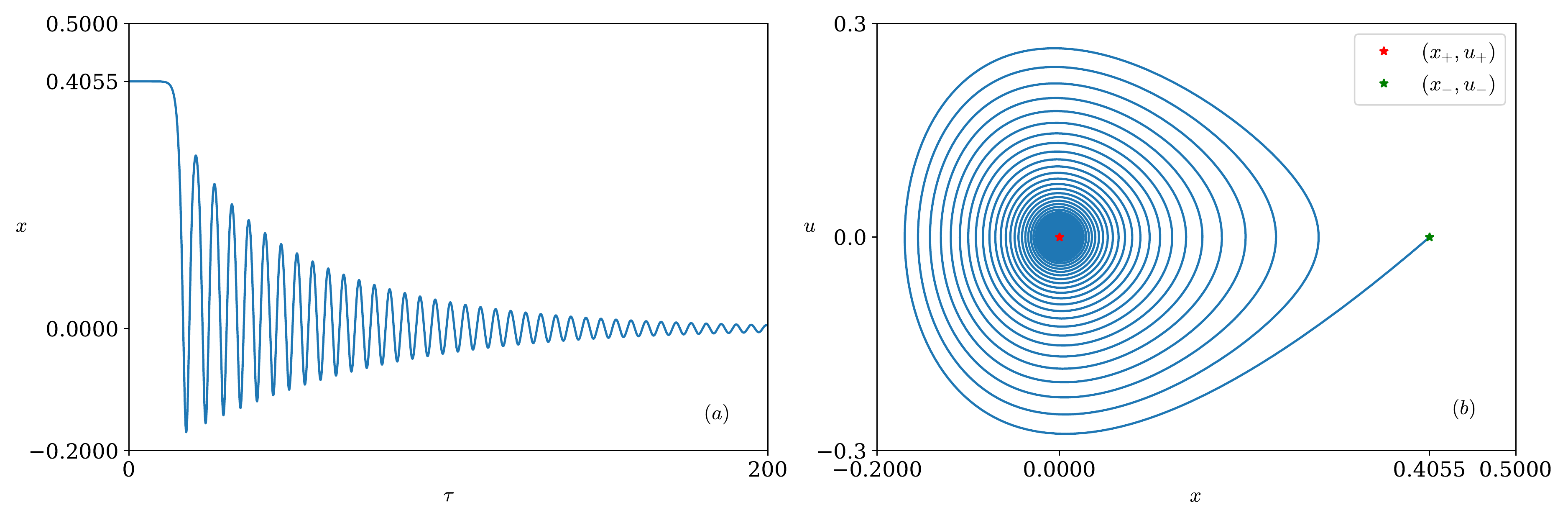}
\caption{(a) Solution $x=\log(\rho)$; (b) Phase portrait ($x$,$u$)  ($c=0.02$, $\gamma=1.4$, $\rho_-=1.5$, $\rho_+=1$) }
\label{fig:PSqhd}
\end{figure} 


\section{Undular bores in a Boussinesq-Peregrine system}
\label{sec:pbs}
Undular bores are structures observed on free surface flows that propagate mainly in one direction.  They have been described in a weakly dispersive and weakly nonlinear 
asymptotic regime by the Korteweg-de Vries-Burgers (KdVB) equation. Recently, it was shown that Peregrine's system \cite{peregrine1967} with weak dissipation 
can also describe undular bores as traveling wave solutions with high accuracy \cite{BMT2022}. Note that Peregrine's theory was initiated for the study of solitary waves and also for the study of undular bores \cite{Pere1966}. 

To illustrate consider a dissipative Peregrine-Boussinesq system written in nondimensional unscaled form
\begin{equation}\label{eq:boussinesqs2b} 
    \begin{aligned}
        &\eta_t + u_x + (\eta u)_x = 0\ ,\\
        &u_t + \eta_x + uu_x - \delta u_{xxt} - \eps u_{xx} = 0\ ,
    \end{aligned}
\end{equation}
where $\eta$ denotes the free-surface elevation above the rest position $\eta=0$, $u$ is the horizontal velocity of the fluid evaluated at some depth $\theta$ above the horizontal bottom located at depth $\theta=-1$, while $\delta, \eps>0$. This system is a dispersive/dissipative extension of the nonlinear shallow-water wave equations (also known as St. Venant equations). The latter form a system of  hyperbolic conservation laws derived as a low-order approximation of  the Euler equations for water wave theory \cite{whitham2011linear}.

Here, we will consider traveling wave solutions of \eqref{eq:boussinesqs2b} propagating with speed $s>1$ to the right. The symmetry $s\to -s$, $\eta\to \eta$, $u\to -u$ implies that anything true for traveling waves with $s>1$ is also valid for traveling waves with $s<-1$. The existence of traveling wave solutions to \eqref{eq:boussinesqs2b} was established in \cite{BMT2022} describe undular bores when
$\eps^2<4\delta s \alpha(s)$ and regularized shock waves when $\eps^2\geq 4\delta s \alpha(s)$, where
$\alpha(s)=\frac{s-\sqrt{s^2+8}}{2}+\frac{4s}{(s-\sqrt{s^2+8})^2}$.
Here we verify that if $\delta=o(\eps)$ as $\eps\to 0$, even in the regime $\eps^2<4\delta s\alpha(s)$, these traveling waves tend to a classical shock waves of the 
shallow water equations. 

In order to apply the previous theory, consider the {\em ansatz} 
$$
\eta^{\eps, \delta}(x,t) = -\eta(\tau), \quad u^{\eps, \delta}(x,t) =  u(\tau), \quad \tau=-\frac{x-st}{\sqrt{s\delta}}\ .
$$
and assume for simplicity that $\lim_{\tau\to-\infty}(\eta,u)=(0,0)$ and $\lim_{\tau\to+\infty}(\eta,u)=(\eta_+,u_+)$. The Rankine-Hugoniot conditions dictate (see \cite{BMT2022})
$$u_+=\frac{3s-\sqrt{s^2+8}}{2}<s\ .$$
After integration over $(-\infty, \tau)$ the system \eqref{eq:boussinesqs2b} yields
\begin{equation}\label{eq:redsysa}
-s\eta-u+\eta u=0, \quad su+\eta-\frac{1}{2}u^2-u''-\frac{\eps}{\sqrt{s\delta}}u'=0\ ,
\end{equation}

Eliminating the unknown $\eta$ in \eqref{eq:redsysa} we obtain the second-order equation
\begin{equation}\label{eq:newsecordeq}
u''+cu'+\phi(u)=0\ ,
\end{equation}
where $c=\eps/\sqrt{s\delta}$ and 
\begin{equation}\label{eq:newpotent}
\phi(u)=-su+\frac{u}{s-u}+\frac{1}{2}u^2\ .
\end{equation}
The potential energy 
$$\Phi(u)=\int_0^u\phi(z)dz=\frac{1}{6}u^3-\frac{s}{2}u^2-u+s\ln \frac{s}{s-u}\ ,$$
has an inflection point at $u_c=s-\sqrt[3]{s}$ a maximum at $(0,0)$ and a minimum $(u_+,\Phi(u_+))$, with $u_+>u_c=s-\sqrt[3]{s}$.
One checks that it satisfies \eqref{hypphi1}--\eqref{hypphi2}, and that $\Phi''(u_-)=\phi'(0)=(1-s)(1+s)/s<0$ for $s>1$, 
while $\Phi''(u_+)=\frac{(u_+-s)^3+s}{(u_+-s)^2}>0$ holds since $u_+>s-\sqrt[3]{s}$ for $s>1$. The graph of $\Phi(u)$ is depicted in Figure \ref{fig:potbous}. 
\begin{figure}[ht!] \centering 
\includegraphics[width=\textwidth,clip,trim=0pt 0pt 0pt 0pt]{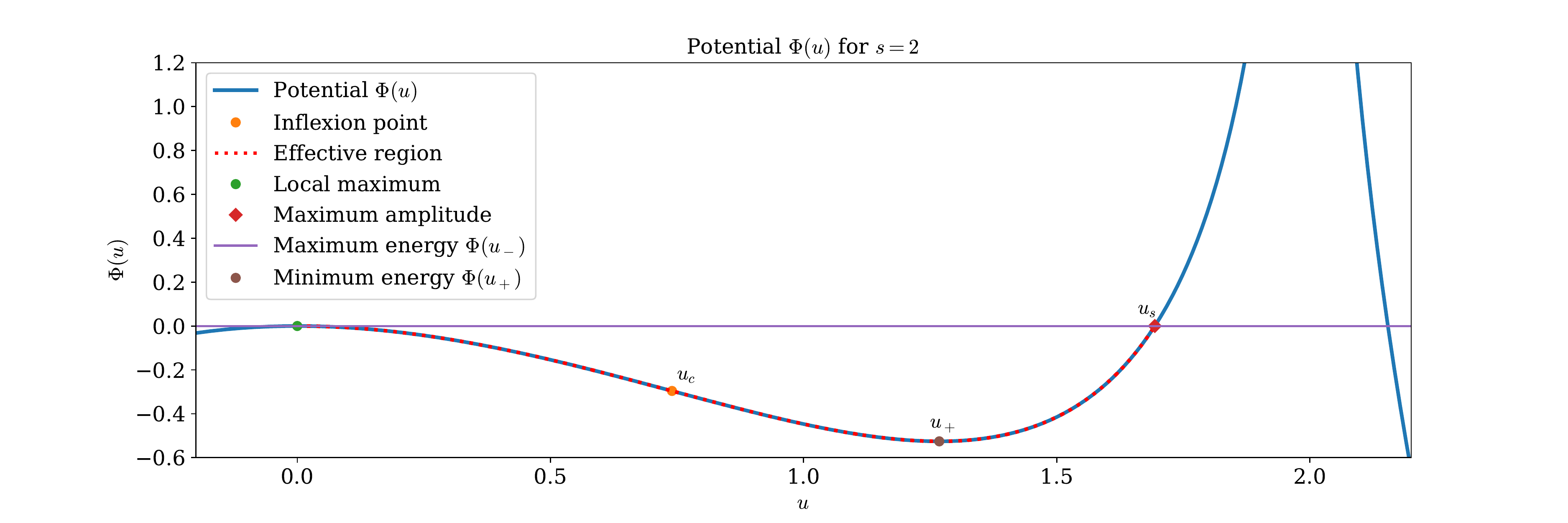}
\caption{The potential $\Phi(u)$ for $s=2$}
\label{fig:potbous}
\end{figure}

Proposition \ref{propsize} in Section \ref{sec:main} may be applied directly to \eqref{eq:newsecordeq} with $c=\eps/\sqrt{s\delta}$. 
It shows that traveling wave solutions of \eqref{eq:boussinesqs2b} tend to entropic shocks of the shallow water wave equations when $\delta=o(\eps)$ 
as $\eps,\delta \to 0$. Figure \ref{fig:convergence} shows the convergence of a dissipative-dispersive shock wave to a classical shock wave obtained numerically by taking $\delta=\eps^{1.5}$ and $s=2$ as $\eps\to 0$. We observe that as $\delta$ becomes smaller the interval where the oscillations are extended becomes smaller as well. The wave-front becomes steeper tending in the limit to a shock. In all cases the quantity $\eps^2-4\delta s\alpha(s)$ remained negative even if it was very small. 

\begin{figure}[hbp] \centering 
\includegraphics[width=\textwidth,clip,trim=0 0 0 0]{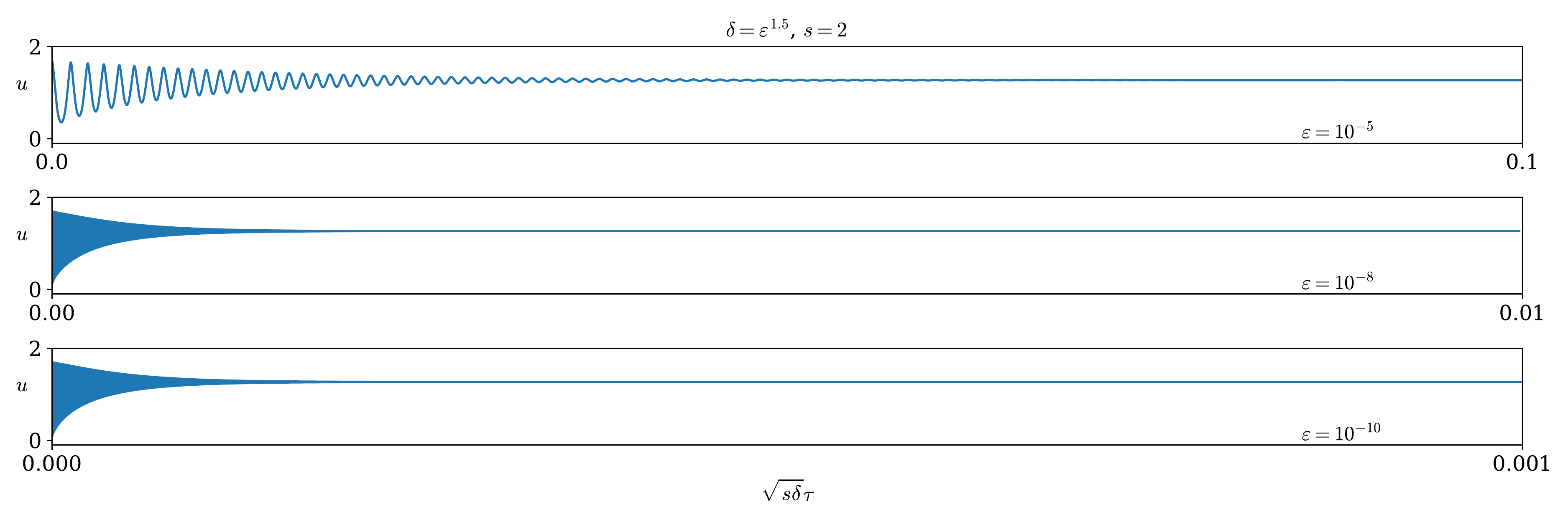}
\caption{Convergence of a dissipative-dispersive shock wave to a classical shock wave of the shallow water wave equations when $\delta=o(\eps)$ as $\delta,\eps \to 0$. (The horizontal axis scales vary between images while the maximum of the traveling waves is at $\tau=0$.)}
\label{fig:convergence}
\end{figure}

\section*{Acknowledgments}

DM thanks KAUST for their hospitality during a visit when this work was initiated.

%
%
%

%
%
%
%


\end{document}